\documentclass[11pt, oneside]{article}
\pdfoutput=1
 \usepackage[text={5.58in,8.5in},centering,letterpaper,dvips]{geometry}
\usepackage{amsmath}
\usepackage[english]{babel}
\usepackage{mathtools}
\usepackage{amssymb}
\usepackage{amsthm}
\usepackage{fancyhdr}
\usepackage{tikz-cd}
\usepackage{url}
\usepackage{xcolor}
\usepackage{enumitem}
\usepackage{titling}
\usepackage{faktor}
\usepackage{subcaption}
\usepackage{caption}
\usepackage{nameref}
\usepackage{hyperref}
\usepackage{float}
\usepackage{graphicx}
\hypersetup{
    pdftitle={Two Sufficient Conditions for a Polyhedron to be (Locally) Rupert}
}

\newtheorem{theorem}{Theorem}[section]
\newtheorem{corollary}[theorem]{Corollary}
\newtheorem{lemma}[theorem]{Lemma}
\newtheorem*{proposition*}{Proposition}
\newtheorem*{theorem*}{Theorem}
\newtheorem{conjecture}{Conjecture}
\newtheorem*{conjecture*}{Conjecture}
\newtheorem*{fact*}{Fact}

\newtheorem*{lemma3ii*}{Lemma 3$(ii)$}
\newtheorem*{lemma3iii*}{Lemma 3$(iii)$}
\newtheorem*{lemma*}{Lemma}

\theoremstyle{definition}
\newtheorem{definition}{Definition}

\newcommand{\R}{\mathbb{R}}
\newcommand{\eps}{\varepsilon}

\DeclareMathOperator{\intr}{Int}
\DeclareMathOperator{\isom}{Isom}

\title{Two Sufficient Conditions for a Polyhedron to be (Locally) Rupert}
\author{Evan J. Scott\thanks{Department of Mathematics, Western Washington Univertisy; scotte31@wwu.edu. Adapted from the author's Senior thesis, which was advised by Jeffrey Meier.}}

\begin{document}
\begin{titlingpage}
\setlength{\droptitle}{30pt} 
\maketitle
\begin{abstract}Given two cubes of equal size, it is possible --- against all odds --- to bore a hole through one which is large enough to pass the other straight through. This preposterous property of the cube was first noted by Prince Rupert of the Rhine\footnote{Count Palatine of the Rhine and Duke of Bavaria, son of Frederick V, the Winter King, Elector Palatine, and
King of Bohemia, and Elizabeth, daughter of James I of England.} in the 17th century. Surprisingly, the cube is not alone --- many other polyhedra have this property, which we call being \textit{Rupert}. A concise way to express that a polyhedron is Rupert is to find two orientations $Q$ and $Q'$ of that polyhedron so that $\pi(Q)$ fits inside $\pi(Q')$, with $\pi$ representing the orthogonal projection onto the $xy$-plane. Given this scheme, to bore the hole in $Q'$ we can remove $\pi^{-1}(\pi(Q))$.

In recent work on this subject, \cite{Jerrard2017}, the authors made the conjecture\footnote{``With a certain hesitancy''}, still open, that every convex polyhedron is Rupert. Aiming at this conjecture, we give two sufficient conditions for a polyhedron to be Rupert. Both conditions require the polyhedron to have a particularly simple orientation $Q$, which we alter by a very small amount to get $Q'$ as required above. When a passage is given by a very small alteration like this, we call it a \textit{local} passage. Restricting to the local case turns out to offer many valuable simplifications. In the process of proving our main theorems, we develop a theory of these local passages, involving an analysis of how small rotations act on simple polyhedra.  \end{abstract}
\end{titlingpage}

\section*{Introduction}

In the middle of the 17th century, Prince Rupert of the Rhine made a bet with a friend that he could achieve the following seemingly impossible feat: Given two cubes of equal size, he claimed to be able to bore a hole through one large enough to pass the other through. His friend, reasonably enough, took Rupert up on this bet. His friend lost. 

In \cite{Jerrard2017}, the authors conjectured that \textit{every} polyhedron has such a passage, that is, that every polyhedron is \textit{Rupert}. The approach to this conjecture thus far has been highly geometric, with focused attention on the details of interesting cases, such as the Platonic and Archimedean solids. We take a different approach, discarding those details to give a more general result.

One way to see that a polyhedron is Rupert is by using the projection $\pi$ onto the $xy$-plane and finding two orientations $Q$ and $Q'$ of the polyhedron so that $\pi(Q)\subseteq\intr(\pi(Q'))$. If this occurs, then we can cut a hole in $Q'$ by removing $\pi^{-1}(\pi(Q))$, which will be large enough to pass $Q$ through by construction. See Figure~\ref{fig:Shadows Prove Cube Rupert} for such a scheme for the Cube: the Cube has two well-known projections, one a square and one a hexagon. The figure shows that the square fits within the hexagon, and hence the Cube is Rupert.

\begin{figure}[h]
\centering
\includegraphics[width = 0.5\linewidth]{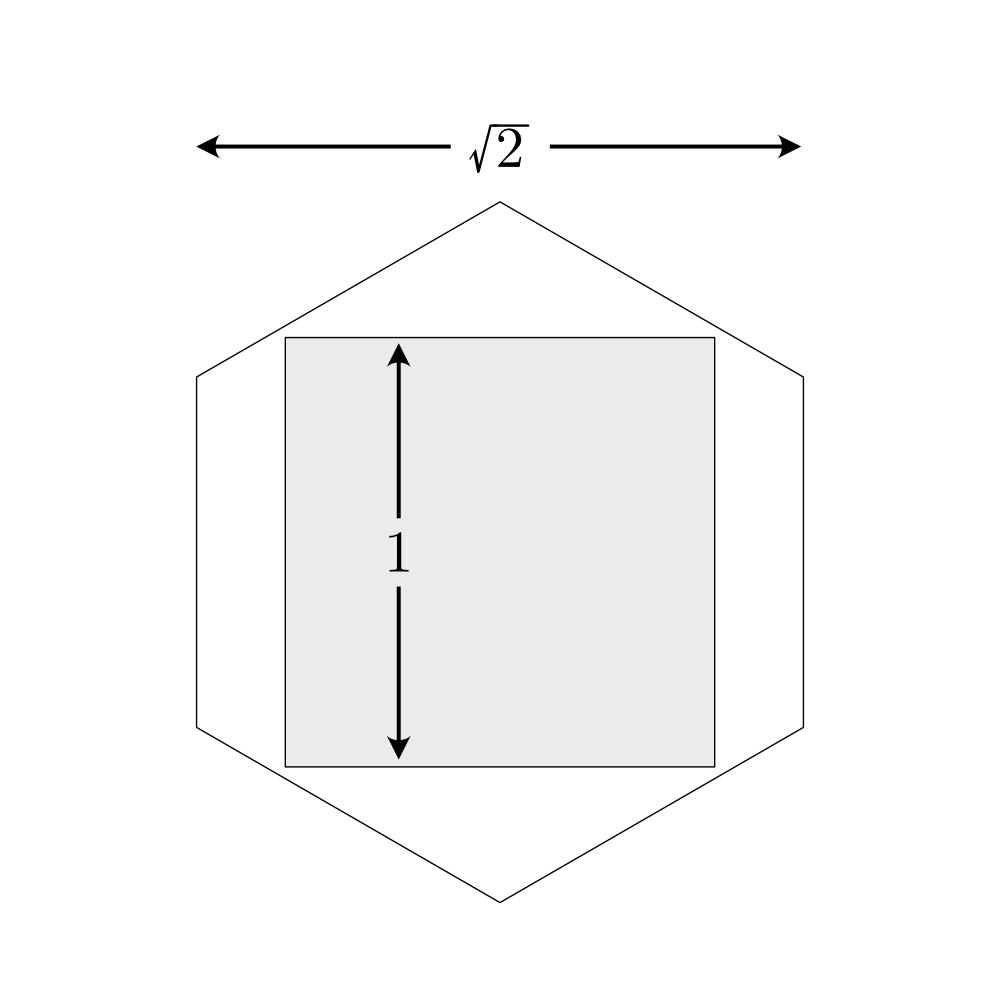}
\caption{Two orthogonal projections of the Cube, a square and a hexagon.}
\label{fig:Shadows Prove Cube Rupert}
\end{figure}

Our main results are two sufficient conditions for a polyhedron to be Rupert, Theorems A and B.

	\begin{theorem*}[A]\label{thm:mainA}
	If $Q$ has a nontrivial double-arch polygonal section $P$, then $Q$ is locally Rupert and hence Rupert.\end{theorem*}
	
	\begin{theorem*}[B]\label{thm:mainB}
	If $Q$ has a prism section $R$ over a polygon $P$ and $P$ is nontrivial double-arch, then $Q$ is locally reverse Rupert and hence Rupert.\end{theorem*}

We'll now give an informal description of these results. For the relevant specific definitions, see Sections 1, 2, and 4. Theorem~\nameref{thm:mainA} and Theorem~\nameref{thm:mainB} give that if a polyhedron $Q$ can be built from a sufficiently nice polygon $P$ by adding vertices above and below $P$ in one of two ways, both ensuring that $\pi(Q) = \pi(P)$, then $Q$ has a Rupert passage. This passage is given by taking a very small rotation $\rho$ so that $\pi(\rho(Q)) \subset \intr(\pi(Q))$ or $\pi(Q) \subset \intr(\pi(\rho(Q)))$ for Theorem~\nameref{thm:mainA} and Theorem~\nameref{thm:mainB}, respectively. These two ``ways'' for building $Q$ from $P$ are demonstrated by the Octahedron and the Cube, see Figure~\ref{fig:Platonic Sections}. 

\begin{figure}[h!]
\centering
	\begin{subfigure}{.5\textwidth}
	\centering
	\includegraphics[height = .75\linewidth]{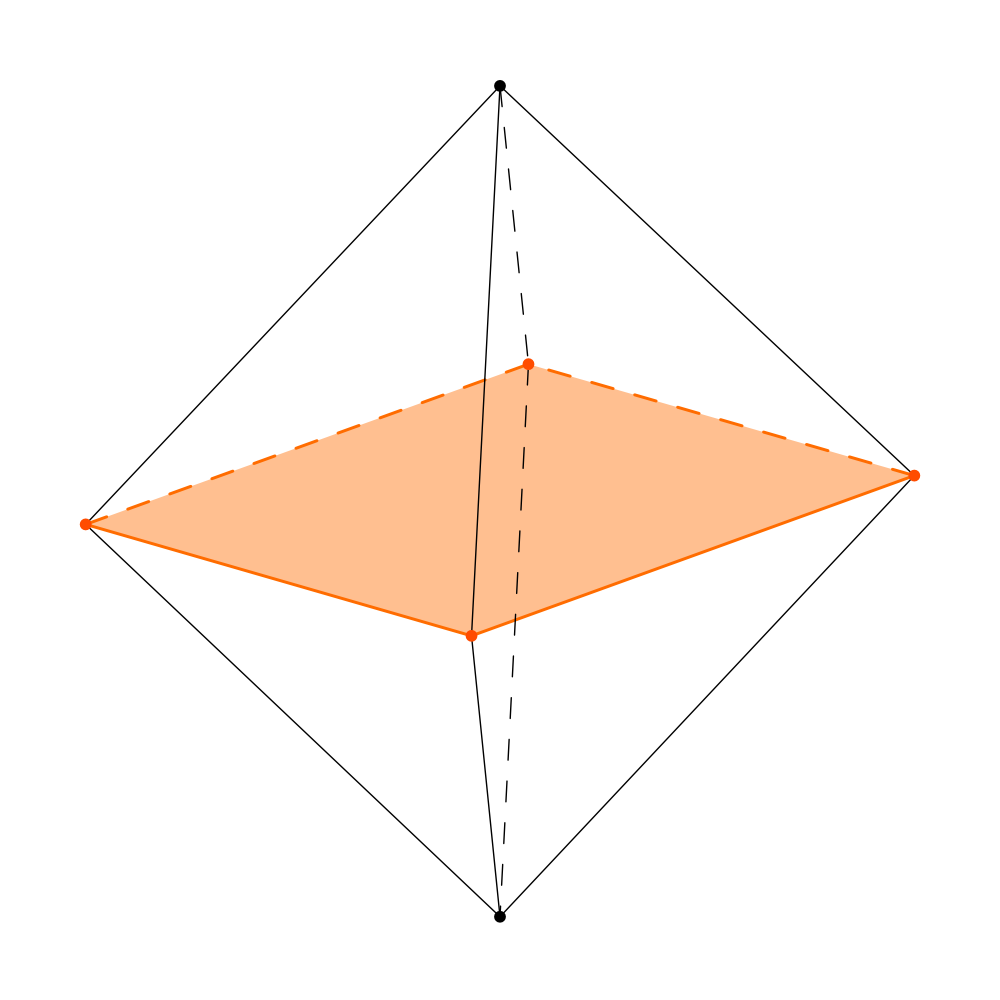}
	\caption{}
	\label{fig:Octahedron Polygonal Section}
	\end{subfigure}%
	\begin{subfigure}{.5\textwidth}
	\centering
	\includegraphics[height = .75\linewidth]{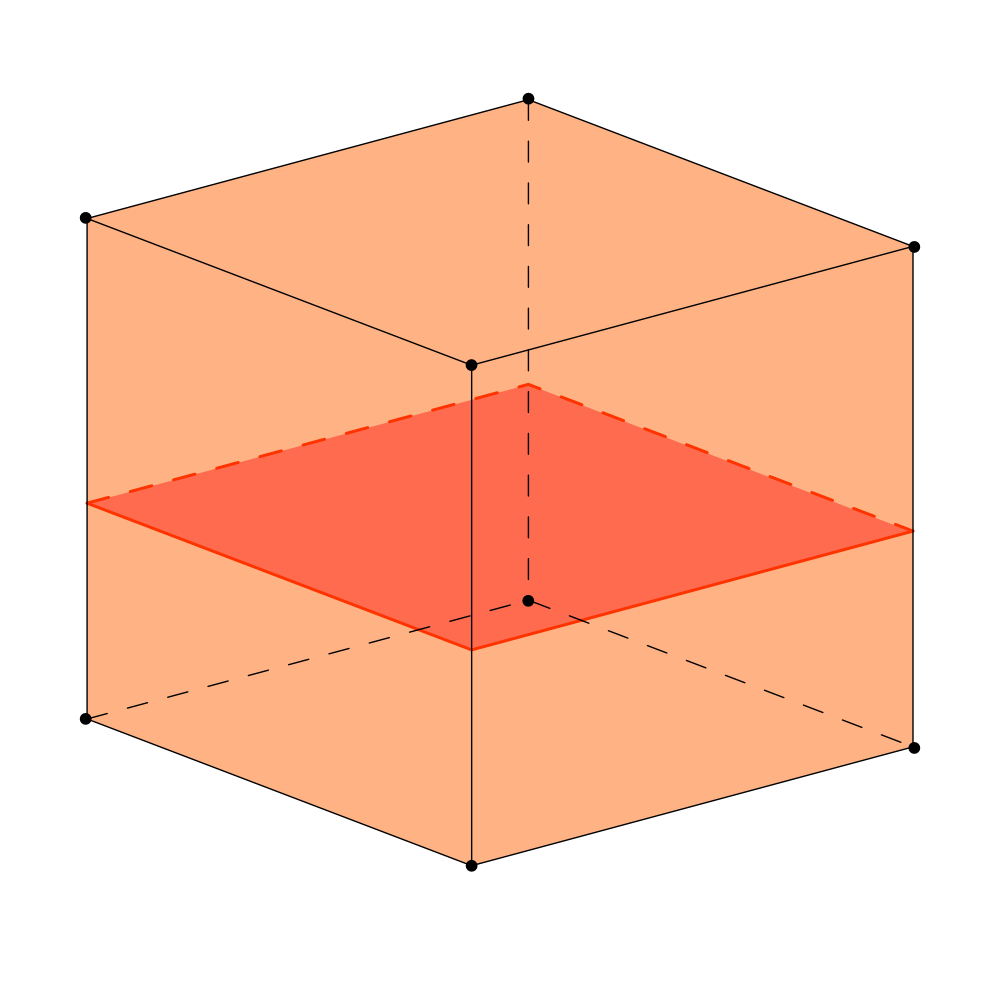}
	\caption{}
	\label{fig:Cube Prism Section}
	\end{subfigure}
\caption{The Octahedron and Cube are both built over squares.}	
\label{fig:Platonic Sections}
\end{figure}

Both are built from a square: the Octahedron can be built by adding a vertex above and a vertex below the interior of the square, and the Cube can be built by extending the square vertically. In this latter case, we can then add vertices anywhere above and below the square, but this isn't necessary for the Cube. In this paper we show that both polyhedra, and many more, are Rupert.

This paper is constructed as follows: in Section~2 we show that by limiting to the case of an arbitrarily small rotation $\rho$, we can restrict our attention only to the simple case of a polyhedron whose vertices all lie in the $xy$-plane. In Section~3 we develop the theory of how small rotations $\rho$ act on these ``flat'' polyhedra, and in Section~4 we apply this theory to prove our main results. In Section~5 we survey which polyhedra we can prove to be Rupert with our main theorems. Finally, in Section~6 we offer a few directions for future work.

\section*{Acknowledgements}

I would primarily like to thank my advisor, Jeffrey Meier, for his wide range of help and support on this paper. I'd also like to thank Zeb Howell for making the figures and Hart Easley for an extremely helpful early proofread. On a personal note, I'd like to thank my family for supporting my pursuit of mathematics, as well as my professor, Anne Hafer, for encouraging me to study this subject to begin with. This paper would not exist without all of these people and more.

\section{Preliminaries}

All polyhedra we consider in this paper are convex. An \textit{oriented polyhedron} $Q$ is the convex hull of a finite set of vertices $V\subseteq \R^3$. An oriented polyhedron has combinatorial data, its vertices, faces, and edges, all given in the usual way. Any convex body in $\R^3$ is the convex hull of its \textit{extreme points}, which do not lie on any straight line connecting two other points. We can thus assume, without loss of generality, that $V$ consists only of extreme points; in other words, each $q\in V$ is a vertex of $Q$ in the combinatorial sense. 

Two oriented polyhedra $Q$ and $Q'$ are \textit{congruent} if there exists a rigid motion of $\R^3$ taking $Q$ to $Q'$. We call an equivalence class under this congruence an \textit{unoriented polyhedron}, and we call a representative of a congruence class $\mathcal Q$ an \textit{orientation of $\mathcal Q$}. In this paper, when the word ``polyhedron'' is used alone, it should be taken to mean ``oriented polyhedron.'' This is nonstandard, but helps to simplify later discussion.

Let $\pi\colon \R^3\to \R^3$ be the orthogonal projection onto the $xy$-plane, $\{(x, y, 0)\in \R^3\}$. An unoriented polyhedron $\mathcal Q$ is \textit{Rupert} if there exists two orientations $Q, Q'$ of $\mathcal Q$ so that $\pi(Q)\subseteq \intr(\pi(Q'))$. If this occurs, we can ``drill a hole straight through'' $Q'$ by taking $Q'\setminus\pi^{-1}(\pi(Q))$. This hole is honestly drilled through the ``body'' of $Q'$, since $\pi(Q)$ is contained in the interior of $\pi(Q')$; furthermore, by construction, we can slide $Q$ through this hole by translating it vertically. Such a scheme is called a \textit{Rupert passage} for $\mathcal Q$, often shortened simply to ``passage.''

The \textit{special orthogonal group $SO(3)$} is the group of orthogonal matrices with determinant $1$ in $GL_3(\R)$. The group $SO(3)$ inherits its smooth action on $\R^3$ from the general linear group, and this action shows that $SO(3)$ is the group of rotations of $\R^3$ fixing the origin. Euler's Rotation Theorem states that every element of $SO(3)$ is a rotation by some angle around some axis, which allows us to represent an element of $SO(3)$ as a unit vector $a$, our axis, and $\theta\in S^1$, our angle. The rotation of $\theta$ around $a$ by the right-hand rule is written $\rho_a^\theta$. 

The unit sphere in $\R^3$ will be written $S$, so that a rotation $\rho_a^\theta$ has $a\in S$. We assume that our spheres have a metric and angle measures given in the usual way. A \textit{great circle} is the intersection of a planar subspace of $\R^3$ with our sphere, and a \textit{(spherical) straight line} is a segment of a great circle. We will appeal to spherical trigonometry later, but no familiarity with the subject is assumed.

The group $SO(3)$ is a 3-dimensional Lie group, so near the identity we inherit the metric space structure from $\R^3$. For rotations $\rho_a^\theta\in SO(3)$ with $|\theta|<\pi$, we define $|\rho_a^\theta| \coloneqq |\theta|$. The value $|\rho_a^\theta|$ coincides with the distance from $\rho_a^\theta$ to the identity under the metric inherited from $\R^3$. 

A rotation $\rho\in SO(3)$ is a \textit{Rupert rotation} for an oriented polyhedron $Q$ if $\pi(\rho(Q))\subseteq \intr(\pi(Q))$. Clearly, if an unoriented polyhedron $\mathcal Q$ has an orientation $Q$ with a Rupert rotation, then $\mathcal Q$ is Rupert. This scheme for finding Rupert passages does not cover all possible passages, but it is structured enough to allow for the development of some theory.

\section{Locally Rupert and Polygonal Sections}

When working with any Lie group action, such that of $SO(3)$ on $\R^3$, there is often much to be gained by analyzing the actions of group elements near the identity. In our situation, this manifests as Rupert rotations with very small angular components, rotations which will change the polyhedron very little. This motivates the following definition.

\begin{definition}[Locally Rupert]
An oriented polyhedron $Q$ is \textit{locally Rupert} if, for all $\eps>0$, there exists a Rupert rotation $\rho$ for $Q$ with $|\rho|< \eps$. An unoriented polyhedron is locally Rupert if it has a locally Rupert orientation.
\end{definition}

Clearly, if a polyhedron is locally Rupert then it is Rupert. We take the local approach because it offers some key simplifications, including the following. 

\begin{definition}[Flat Polygon]
A \textit{flat polygon} $P$ is a polyhedron formed as the convex hull of vertices $V\in R^3$, where each vertex $v\in V$ lies in the $xy$-plane. More simply, a flat polygon is a convex polygon embedded into the $xy$-plane in $\R^3$. We will use $\partial P$ for a flat polygon $P$ to mean its topological boundary as a subset of the plane, rather than as a subset of $\R^3$.
\end{definition}

For this next definition, see Figure~\ref{fig:Polygonal Section Definition}.

\begin{figure}[h]
\centering
\includegraphics[width = 0.5\linewidth]{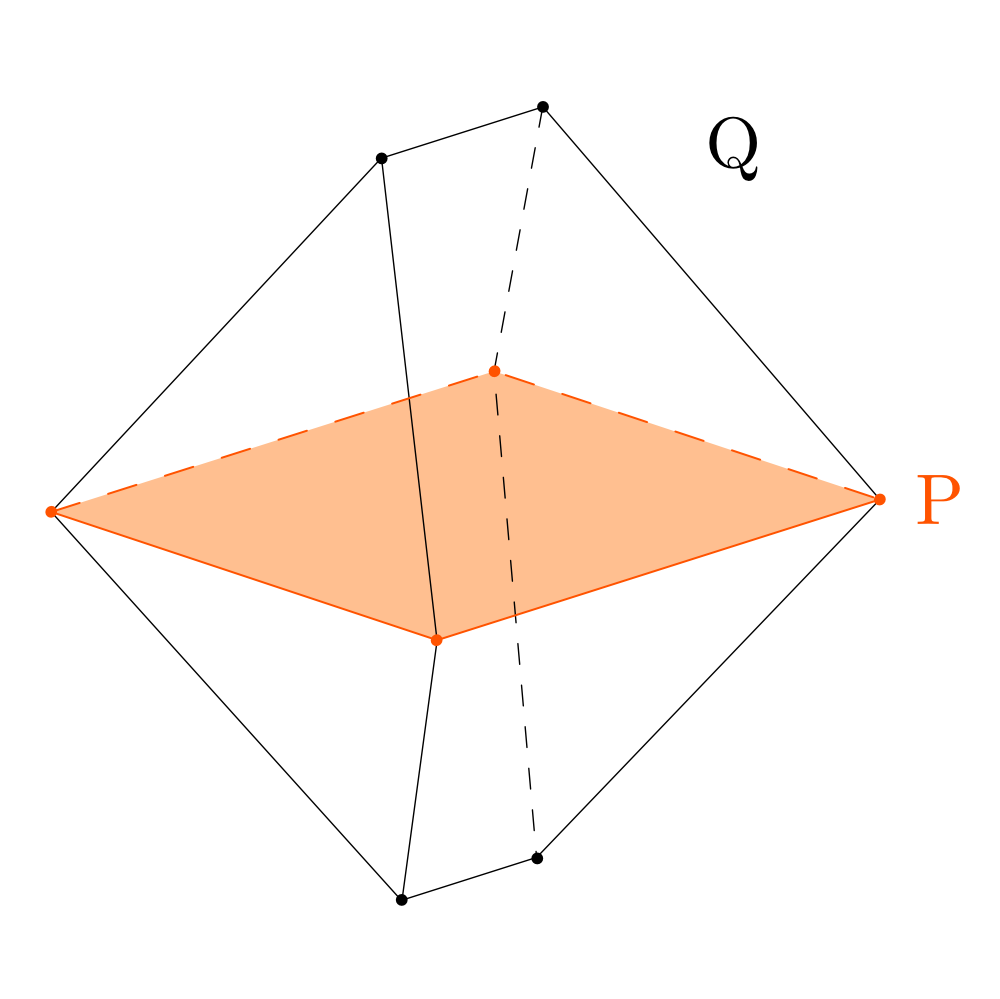}
\caption{An example of a square polygonal section.}
\label{fig:Polygonal Section Definition}
\end{figure}

\begin{definition}[Polygonal Section]
A \textit{polygonal section} of a polyhedron $Q$ is a flat polygon $P$ so that 
\begin{enumerate} 
\item $P$ is the intersection of $Q$ and the $xy$-plane, and
\item $\pi^{-1}(\partial P) \cap Q = \partial P$.
\end{enumerate}
\end{definition}

As a first observation, we can deduce that all vertices of $Q$ other than the vertices of $P$ are in $\intr(\pi^{-1}(P))$. See Figure~\ref{fig:Projection Equality of Polygonal Sections} for the following discussion. The vertices of $Q$ lying in the $xy$-plane must be in $P$ by $1$, and if a vertex $q$ of $Q$ lies outside of the $xy$-plane and $\intr(\pi^{-1}(P))$, the convex hull of that vertex and the vertices of $P$ would intersect $\pi^{-1}(\partial P)$ outside of $\partial P$, contradicting 2. From this we can conclude that $\pi(Q) = \pi(P) = P$. 

\begin{figure}[h]
\centering
\includegraphics[width = 0.5\linewidth]{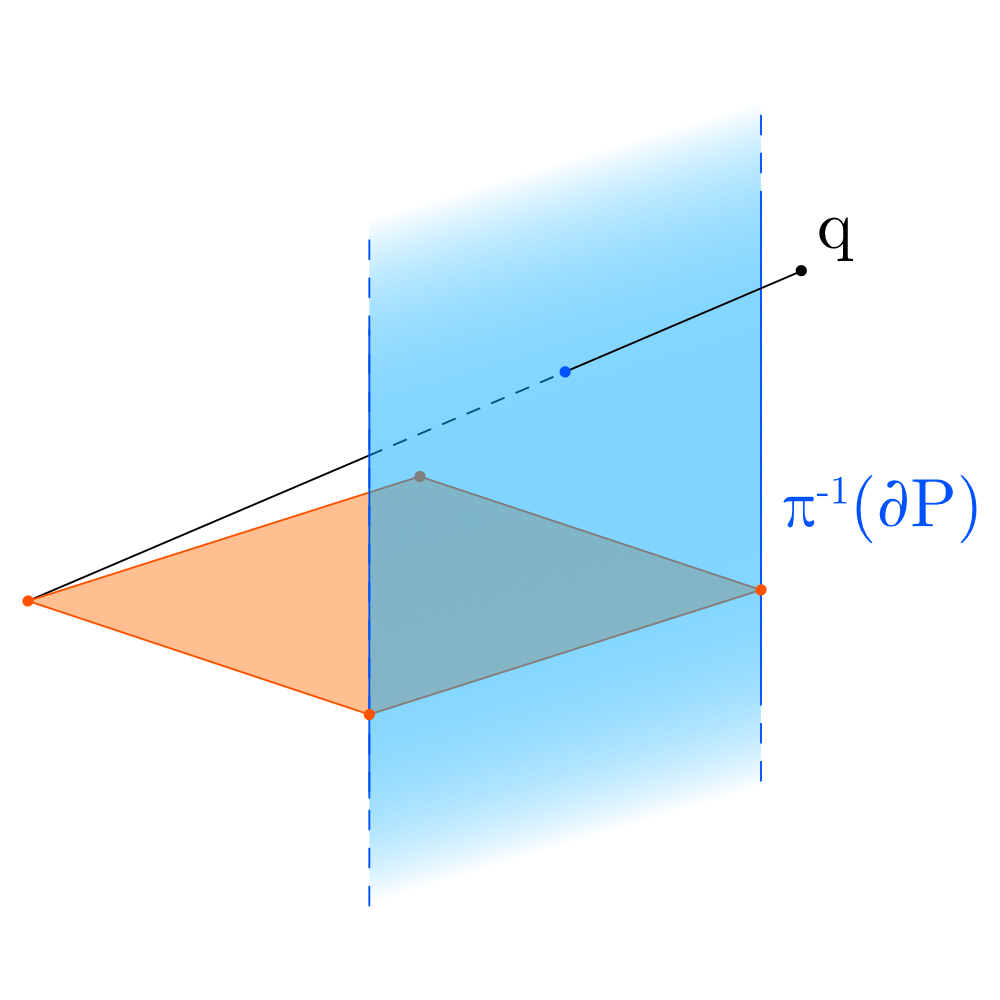}
\caption{A vertex $q$ of $Q$ which is outside $\pi^{-1}(P)$ contradicts 2.}
\label{fig:Projection Equality of Polygonal Sections}
\end{figure}


Not every polyhedron $Q$ has a polygonal section; for example, the Cube does not. When $Q$ does have a polygonal section, say $P$, there is a close relationship between the locally Rupert property of $P$ and that of $Q$. For each vertex $v$ of $Q$ that doesn't lie on $P$, since $v\in \intr(\pi^{-1}(P))$, $\pi(v)$ lies in the interior of $P = \pi(Q)$. Intuitively, this leaves a little ``wiggle room'' for $v$ (and hence, $\pi(v)$) to move around without affecting the shape of $\pi(Q)$. Thus, for small enough rotations, the only vertices that impact the shape of $\pi(Q)$ are those lying on $P$, so if $P$ is locally Rupert, $Q$ should be as well.

\begin{lemma}[Polygonal Section Bootstrap]\label{lem:psbs}
Let $Q$ be a polyhedron with a polygonal section $P$. If $P$ is locally Rupert, then $Q$ is locally Rupert.
\end{lemma}

\begin{proof}
The continuous action of $SO(3)$ on $\mathbb R^{3}$ gives an action function $SO(3)\times \R^3\to \R^3$ given by $(g,x)\mapsto g\cdot x$. Restricting this function to the subspace $SO(3)\times \{a\}$ for some $a\in \R^3$ and composing with $\pi$ gives a continuous function $\psi_a\colon SO(3)\to \R^3$ given by $\psi_a(g) = \pi(g\cdot a)$. The function $\psi_a$, like $\pi$, maps into the $xy$-plane.

Let $H$ be the set of vertices of $Q$ not lying on $P$. For each vertex $h\in H$, $\pi(h)\in \intr(\pi(Q))$, since $h\in \intr(\pi^{-1}(P))$. Let $K_{\pi(h)} = \psi_h^{-1}(\intr(\pi(Q))$. Explicitly, $K_{\pi(h)}$ is the set of rotations $\sigma\in SO(3)$ so that $\pi(\sigma\cdot h)\in \intr(\pi(Q))$. Letting $I$ be the identity in $SO(3)$, we see that $\pi(I\cdot q)\in  \intr(\pi(Q))$ and hence $I\in K_{\pi(h)}$. Let $K = \bigcap_{h\in H} K_{\pi(h)}$. The set $K$ is open since $H$ is finite, and $I\in K$ since $I$ is in all the component sets of the intersection.

Near the identity, $SO(3)$ is a metric space, so since $K$ is open and contains the identity, there is some radius $\eta>0$ so that the open ball of radius $\eta$ centered at the identity in $SO(3)$ is contained within $K$. Let $\eps>0$ be given. We wish to find a Rupert rotation $\rho$ for $Q$ so that $|\rho|<\eps$. Let $\delta = \min\{\eps, \eta\}$. Since $P$ is locally Rupert, there exists a Rupert rotation $\rho$ for $P$ with $|\rho|<\delta$, and hence $|\rho|<\eps$ and $|\rho|<\eta$. We claim that $\rho$ is a Rupert rotation for $Q$.

To show that $\rho$ is a Rupert rotation for $Q$, by convexity it suffices to show that for each vertex $v$ of $Q$, $\pi(\rho\cdot v)\in \intr(\pi(Q))$. There are two cases to handle - either $v$ is a vertex of $P$ or $v\in H$. If $v$ is a vertex of $P$, since $\rho$ is a Rupert rotation for $P$ we see that $\pi(\rho\cdot v)\in \intr(\pi(P))$, but since $\pi(P) = \pi(Q)$ this case is done. 

If $v\in H$, then since $|\rho|<\eta$, $\rho\in K$ and hence $\rho\in K_{\pi(v)}$. By definition then, $\pi(\rho\cdot v)\in \intr(\pi(Q))$, and we are done.

Thus, since $\pi(\rho\cdot v)\in \intr(\pi(Q))$ for all vertices $v$ of $Q$, by taking the convex hull on the left we get $\pi(\rho(Q))$, and since $\intr(\pi(Q))$ is convex, thus $\pi(\rho(Q))\subseteq\intr(\pi(Q))$ and $\rho$ is a Rupert rotation for $Q$ with $|\rho|<\eps$, as desired.
\end{proof}

Note that the converse is fairly simple, following from the facts that $P\subseteq Q$ and $\pi(Q) = \pi(P)$. This lemma tells us that, in the local case, finding Rupert rotations for polyhedra $Q$ with a polygonal section $P$ restricts to finding Rupert rotations for $P$.

\section{Allowable Sets and the $J$ Function}

We now aim to understand when a flat polygon $P$ is locally Rupert. For the following discussion, we fix a small angle $\delta>0$. Recall that $S$ is the unit sphere in $\R^3$, where our axes in the axis-angle representation of rotations live.

\begin{definition}[Allowable Axis Set $A_v^\delta$]
Let $P$ be a flat polygon, and let $v$ be a vertex of $P$. The \textit{allowable axis set} for $v$ with rotation amount $\delta$, $A_v^\delta$, is \[\{a\in S \mid \pi(\rho_a^\delta(v))\in\intr(\pi(P))\}\] Informally, this is the set of ``good'' axes for $v$ with rotation amount $\delta$.
\end{definition}

A flat polygon $P$ is locally Rupert if and only if, for all $\eps>0$, there exists a $0<\delta<\eps$ so that $\bigcap_{v\in P} A_v^\delta\not=\varnothing$: an axis $a\in \bigcap_{v\in P} A_v^\delta$ has the property that $\pi(\rho_a^\delta(v))\in \intr(\pi(P))$ for all vertices $v$ of $P$. By taking the convex hull of $\pi(\rho_a^\delta(v))$ across all vertices $v$, we get $\pi(\rho_a^\delta(P))\subseteq \intr(\pi(P))$, so we see that $\rho_a^\delta$ is a Rupert rotation by definition. Since $|\rho_a^\delta| = \delta$, and $\delta<\eps$, this shows that $P$ is locally Rupert. The other direction is trivial.

In order to study this problem, we give an explicit calculation of the shape of the allowable sets $A_v^\delta$. We will restrict our attention to vertices that lie on the $xy$-plane, since that is the case for flat polygons.\footnote{Note that much of the following development can be easily adapted to the case of an arbitrary vertex, which might offer a path forward in other cases than that of the flat polygon.}

Let $P$ be a flat polygon, let $v\in P$, and choose some small $\delta$. Let $S_v$ be the sphere in $\R^3$ centered at the origin which passes through $v$, so $v\in S_v$. In $\R^3$, we will say that the positive part of the $z$ axis is ``up,'' and the negative part is ``down.'' We will sometimes refer to the intersection of the positive $z$-axis and $S_v$ as the \textit{positive  pole} of $S_v$.

Let $T_v$ be the subset of $S_v$ given by $S_v\cap \pi^{-1}(\intr(P))$. The set $T_v$ is open on $S_v$ in the subspace topology, since it is the intersection of an open set and our subspace. See Figure~\ref{fig:Tv}. By definition, an axis $a\in S$ is an allowable axis for $v$ if and only if $\rho_a^\delta(v)\in T_v$.

\begin{figure}[h]
\centering
\includegraphics[width = 0.5\linewidth]{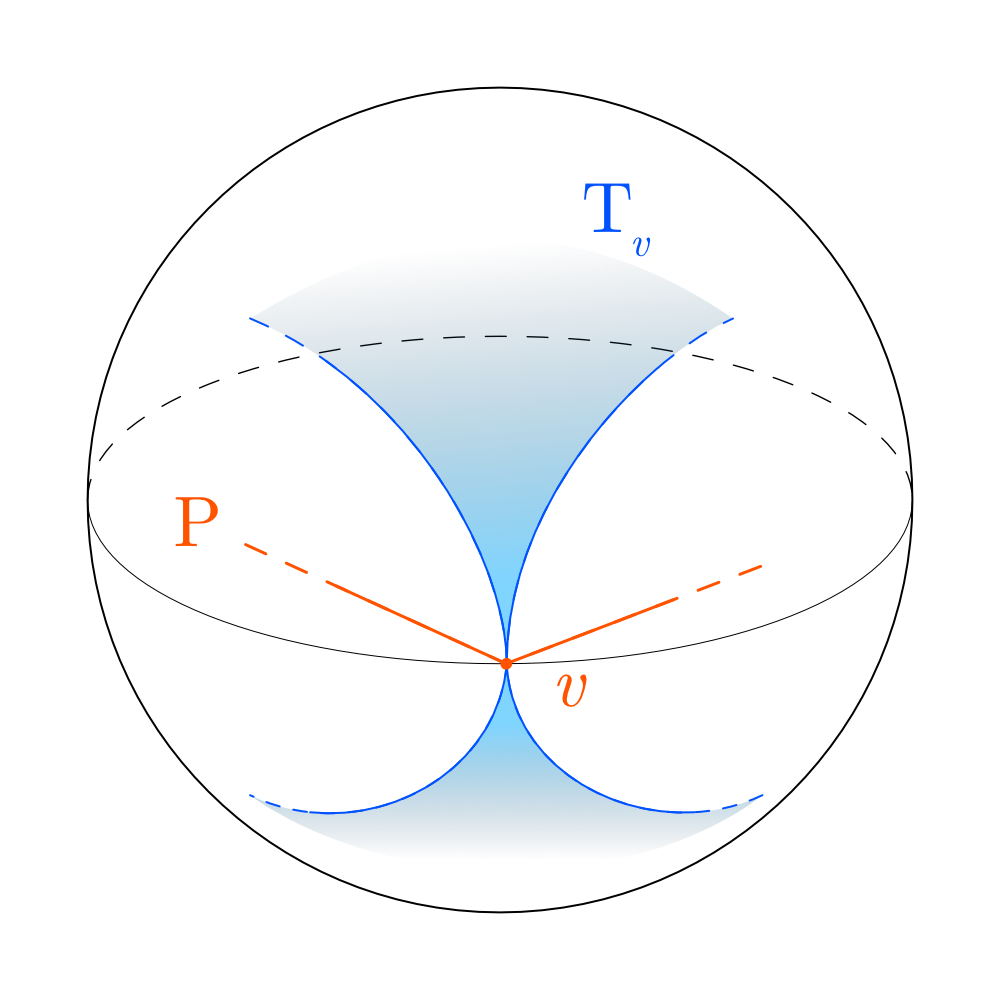}
\caption{An example of a set $T_v = S_v\cap \pi^{-1}(\intr(P))$.}
\label{fig:Tv}
\end{figure}

In order to calculate the shape of $A_v^\delta$, we will construct it using a particular function, and then establish properties of that function. To do this, we now need an important trick: we can naturally identify $S_v$ with the unit sphere $S$ by scaling, since the important information about an axis for rotation is its direction, not its magnitude. This allows us to talk about the allowable set $A_v^\delta$ as a subset of the sphere $S_v$.

\begin{definition}[The map $J_v$]
For a vertex $v\in P$, let $J_v\colon S_v\to S_v$ be defined by mapping $a\in S_v$ to $\rho_a^\delta(v)$.
\end{definition}

\begin{lemma}
$J_v$ is continuous.
\end{lemma}

\begin{proof}
This proof is similar to the construction of $\psi_a$ in Lemma~\ref{lem:psbs}. The axis-angle representation of $SO(3)$ gives that, for nonzero rotations which have a small angle of rotation (say less than some $0<\alpha<\pi$), $SO(3)$ admits a decomposition as $S\times (0, \alpha)$, where $(0, \alpha)$ is the open interval, and the point $(a, \theta)$ in the product corresponds to $\rho_a^\theta\in SO(3)$. Since $S$ is equivalent to $S_v$ by scaling, we get that $SO(3)$ decomposes as $S_v\times (0, \alpha)$.

The action of $SO(3)$ on $S_v$ induces a continuous function $SO(3)\times S_v \to S_v$. Considering $SO(3)$ to be decomposed in the above way, $J_v$ can be constructed by restricting the action function to the subspace $\left(S_v \times (0, \alpha)\right) \times S_v\to S_v$, then further to the subspace $S_v\times \{\delta\}\times \{v\}$, giving $J_v\colon S_v\to S_v$ defined by $J_v(a) = \rho_a^\delta(v)$. This is continuous as it is the restriction of a continuous function to a subspace.
\end{proof}

Let $a\in S_v$. If $J_v(a)\in T_v$, then $\rho_a^\delta(v)\in T_v$ by definition, so this is equivalent to saying $a\in A_v^\delta$. Thus, $A_v^\delta = J_v^{-1}(T_v)$. Since $T_v$ is open, this shows that the allowable sets are open. We would like to know more about $A_v^\delta$, though, so how do we access this information via $J_v$? 

\subsection{Fibering the Sphere}

The surface of the earth has a natural decomposition as a ``product'': each point $p$ has a latitude and a longitude. The latitude is the spherical distance from $p$ to the north pole. The longitude is more complex: start with the $0^\circ$ line, defined to be the (spherical) straight line from the north pole to the south pole passing through Greenwitch, England. This line is called the prime meridian. Now take the circle of points $C$ which have the same latitude as $p$, and measure the angle around $C$ from the intersection of $C$ with the prime meridian to $p$. This is the longitude of $p$. This construction allows us to uniquely represent any point on the Earth, besides the poles, as $(lat, long)$. 

On $S_v$, we can get a similar decomposition by letting $v$ act as the ``north pole.'' See Figure~\ref{fig:Meridian}. To make this work we need to choose a $0^\circ$ or meridian line from $v$ to its antipode. The vertex $v$ lies on the equator of $S_v$, which serves as a natural choice for a meridian line. When looking down on the sphere $S_v$ from above, let the meridian line be the part of the equator which is counterclockwise away from $v$.\footnote{That is, positively oriented relative to the positive pole of $S_v$.} Call this line $M_v$. 

\begin{figure}[h]
\centering
\includegraphics[width = 0.5\linewidth]{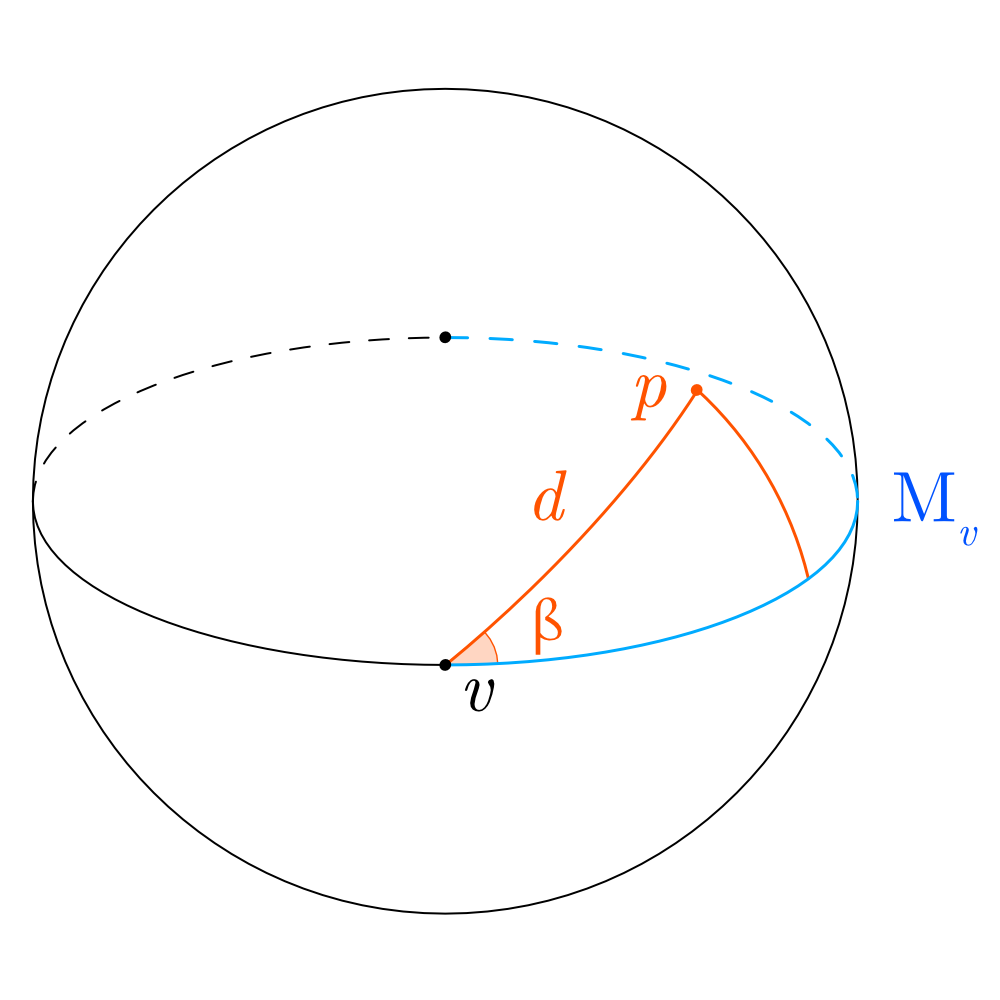}
\caption{The meridian $M_v$ for a vertex $v$ and a point $p = (d, \beta)$.}
\label{fig:Meridian}
\end{figure}

For this figures in this paper, we will always draw $S_v$ with $v$ on the side of the sphere ``facing us.'' This means that $M_v$ is the portion of the equator to the right of $v$ on the figure. This gives, for each point $p\in S_v$, that $p = (d, \beta)$, where $d$ is the spherical distance from $p$ to $v$, and $\beta$ is the ``longitude'' relative to the meridian $M_v$, measured in the same way as for longitude on the Earth. This representation is degenerate only when $p$ is $v$ or the antipode of $v$.	 

\subsection{Fibering $J_v$}

Let $0\leq d\leq2\pi$ be some spherical distance. Let $C_d$ be the circle of radius $d$ around $v$ on $S_v$. Pick some axis $a\in C_d$, and let $t$ be the distance from $v$ to $\rho_a^\delta(v)$. 

This distance $t$ is independent of our choice of $a$, which we show by applying some spherical trigonometry. See Figure~\ref{fig:Independence of t(d) from a} for the triangle $\triangle (a,v,\rho_a^\delta(v))$. We know that for any choice of $a$, the distances from $a$ to $v$ and $a$ to $\rho_a^\delta(v)$ are both equal to $d$, and the angle at $a$, i.e. $\angle(v, a, \rho_a^\delta(v))$, is $\delta$. This is enough to uniquely determine the remaining sides and angles (including $t$) in this triangle by spherical side-angle-side. For $0\leq d\leq 2\pi$, let $t(d)$ be the last side length in the unique spherical triangle defined by two edges of length $d$ meeting at angle $\delta$.

\begin{figure}[h!]
\centering
\includegraphics[width = 0.5\linewidth]{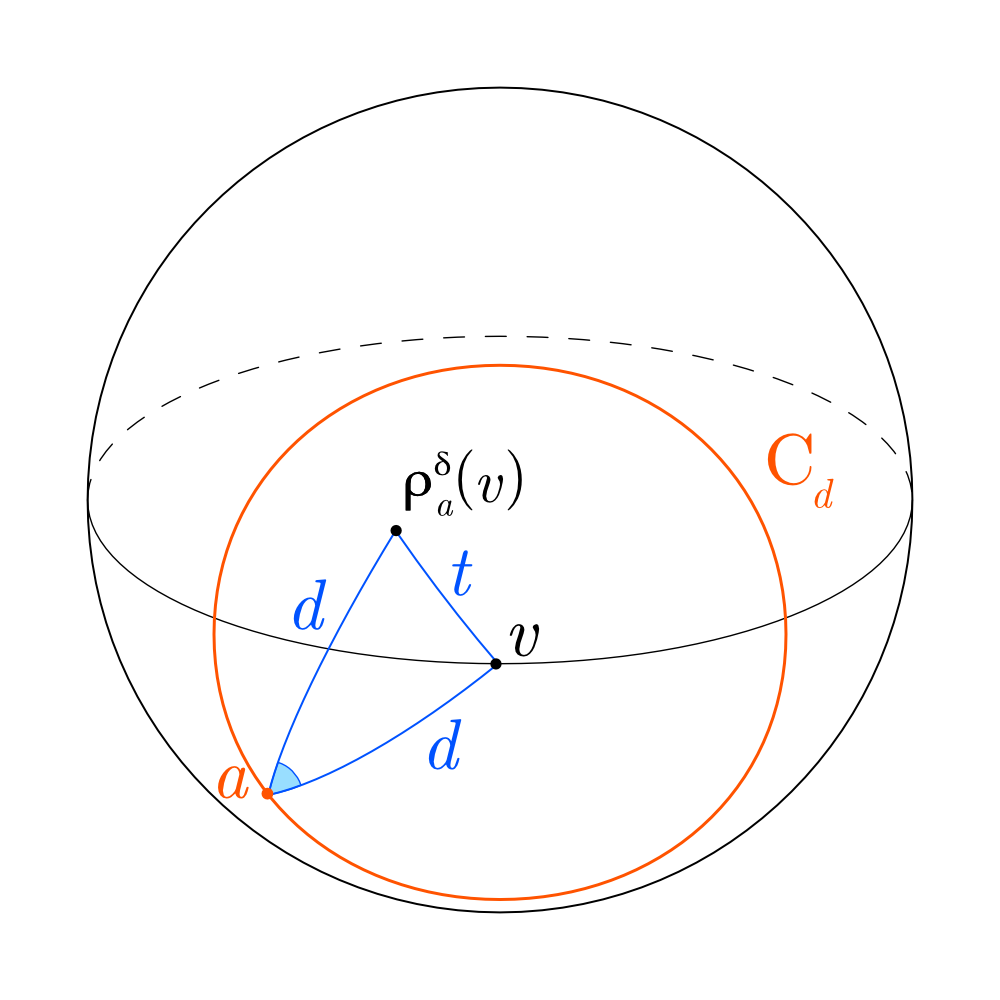}
\caption{The triangle $\triangle (a,v,\rho_a^\delta(v))$ on the sphere.}
\label{fig:Independence of t(d) from a}
\end{figure}

Let $C_{t(d)}$ be the circle of radius $t(d)$ around $v$ on $S_v$. Since $\rho_a^\delta(v)\in C_{t(d)}$ for all $a\in C_d$, this gives us the fact that $J_v|_{C_{d}}$, which we will write $j_{v;d}$ for simplicity, has codomain $C_{t(d)}$. This tells us that $J_v$ behaves nicely with respect to the ``latitude'' relative to $v$. To study $J_v$, we study $j_{v;d}$ for each $d$, then ``glue'' that knowledge together across all $d$ to get the behaviour of $J_v$ on the whole sphere. 

\subsection{Functions on Circles and $j_{v;d}$}

Since $j_{v;d}$ is a function from a circle to another circle, we can draw on the theory of the circle group, $S^1$, to help understand $j_{v;d}$. Let $S^1= \faktor{\R}{2\pi\mathbb Z}$, endowed with the usual metric, act on itself by rotation as a Lie group with identity $0$. Orient $S^1$ positively, i.e. let $S^1$ inherit its orientation from the positive orientation on $\R$.

We'll also need a little bit of the theory of group actions. Let $X$ and $Y$ be sets acted upon by a group $G$. A function $f\colon X\to Y$  is called \textit{$G$-equivariant} if, for all $g\in G$ and $x\in X$, $f(g\cdot x) = g\cdot f(x)$.

\begin{lemma}
Let $f\colon S^1\to S^1$ be any function. If $f$ is $S^1$-equivariant with respect to the action of $S^1$ on itself, then $f$ is an orientation-preserving bijective isometry, i.e. a rotation.
\end{lemma}

\begin{proof}
$S^1$ acts on itself by isometries, so for any $x, a, b\in S^1$, $d(a,b) = d(x\cdot a, x\cdot b)$. We'll write $d(0, x)$ as $|x|$. We first claim that, for all $g\in S^1$ and $x\in S^1$, $d(x, g\cdot x) = |g|$. To see this, apply $x^{-1}$; since $S^1$ is commutative we can pull the $x^{-1}$ inside the term $(g\cdot x)$: \[d(x, g\cdot x) = d(0, x^{-1}\cdot(g\cdot x)) = d(0, g\cdot0) = d(0, g) = |g|\text{.}\] 

We'll now prove that $f$ is an isometry. Let $x, y\in S^1$. We wish to show that $d(x, y) = d(f(x), f(y))$. Since $S^1$ is transitive, let $y = h\cdot x$ for some $h\in S^1$. This gives $d(x, y) = d(x, h\cdot x) = |h|$ and \[d(f(x), f(y)) = d(f(x), f(h\cdot x)) = d(f(x), h\cdot f(x)) = |h|\text{,}\] with the middle equality following from the hypothesis that $f$ commutes with $h$. Thus $d(x, y) = d(f(x), f(y))$ as desired.

Let $\isom(S^1)$ be the group of bijective isometries of the circle. Since $f$ is an isometry, it is injective, so to show that $f\in\isom(S^1)$, we need only prove surjectivity. Let $y\in S^1$. Since $S^1$ is transitive, there exists some $x\in S^1$ so that $y = x\cdot f(0)$. By equivariance, $x\cdot f(0) = f(x\cdot 0) = f(x) = y$, and thus $f$ is surjective.

Since $f\in\isom(S^1)$, it is either a rotation or a flip. The rotations are orientation-preserving and the flips are orientation-reversing. Suppose for the sake of contradiction that $f$ is a flip. Let $\phi$, $\phi'$ be the two fixed points of $f$ and let $x$, $x'$ be the two points exactly half-way between the two fixed points of $f$.
By construction, $f(x) = x'$. Let $\theta$ be the rotation taking $\phi$ to $x$. \[f(\theta\cdot \phi) = f(x) = x'\] but \[\theta\cdot f(\phi) = \theta\cdot\phi = x\] and $x\not=x'$, contradiction. Thus $f$ is a rotation, as desired.
\end{proof}

\begin{corollary}\label{cor:ci}
Let $A, B, C$ be Lie groups isomorphic to $S^1$ with Lie group isomorphisms $I_A, I_B,I_C$ from $A, B, C\to S^1$ which are orientation-preserving isometries. Let $C$ act on $A$ and $B$ by rotation, so for $c\in C$ and $a\in A$, $c\cdot a \coloneqq I_A^{-1}(I_C(c))\cdot a$. The action on $B$ is defined similarly. Let $f\colon A\to B$. If $f$ is $C$-equivariant, then $f$ is an orientation-preserving bijective isometry.
\end{corollary}

\begin{proof}
The corollary follows from application of the isomorphisms at the appropriate time in the above proof. The distance $d(x, g\cdot x) = d(x, I_A^{-1}(I_C(g))\cdot x) = |I_A^{-1}(I_C(g))|$ is still invariant under changing $x$. The fact that the isomorphisms are orientation-preserving is required in the second half of the proof, since otherwise the isomorphisms could ``flip $S^1$ over'' with respect to $C$.
\end{proof}

We need one more ingredient, a statement about $SO(3)$.

\begin{theorem}[$SO(3)$ Conjugation Rule]\label{thm:scr}
For all axes $a, b$ and rotation amounts $\theta, \gamma$, we have \[\rho^\gamma_{\rho_b^\theta(a)} = \rho^\theta_b\circ\rho_a^\gamma\circ\rho_b^{-\theta}\text{.}\]
\end{theorem}

The rule follows from the fact that we can think of $SO(3)$ as a matrix group with conjugation representing change-of-basis, as well as Euler's rotation theorem. For the material required for this proof, see \cite{Tapp2016}.

We apply these theorems to understand the structure of $j_{v;d}$.

\begin{lemma}
The function $j_{v;d}\colon C_d\to C_{t(d)}$ is an orientation-preserving bijective isometry for all $0\leq d\leq 2\pi$.
\end{lemma}

\begin{proof}
In order to apply Corollary~\ref{cor:ci} above, we need to find our isomorphisms. For the following, see Figure~\ref{fig:Construction of Id}. Since $C_d$ and $C_{t(d)}$ are centered at a point on the equator, the equator intersects each circle twice. For each circle, let the intersection with the meridian $M_v$ be $k_d$ and $k_t$ respectively. Naturally identify $C_d$ and $C_{t(d)}$ with $S^1$ by taking a point on either circle to its angle by the right hand rule around $v$ away from $k_i$, defining isomorphisms $I_d, I_t\colon C_d, C_{t(d)} \to S^1$ with $I_d(k_d) = I_t(k_t) = 0$. These maps represent the ``longitude'' in our metaphor from before. 

\begin{figure}[h]
\centering
\includegraphics[width = 0.5\linewidth]{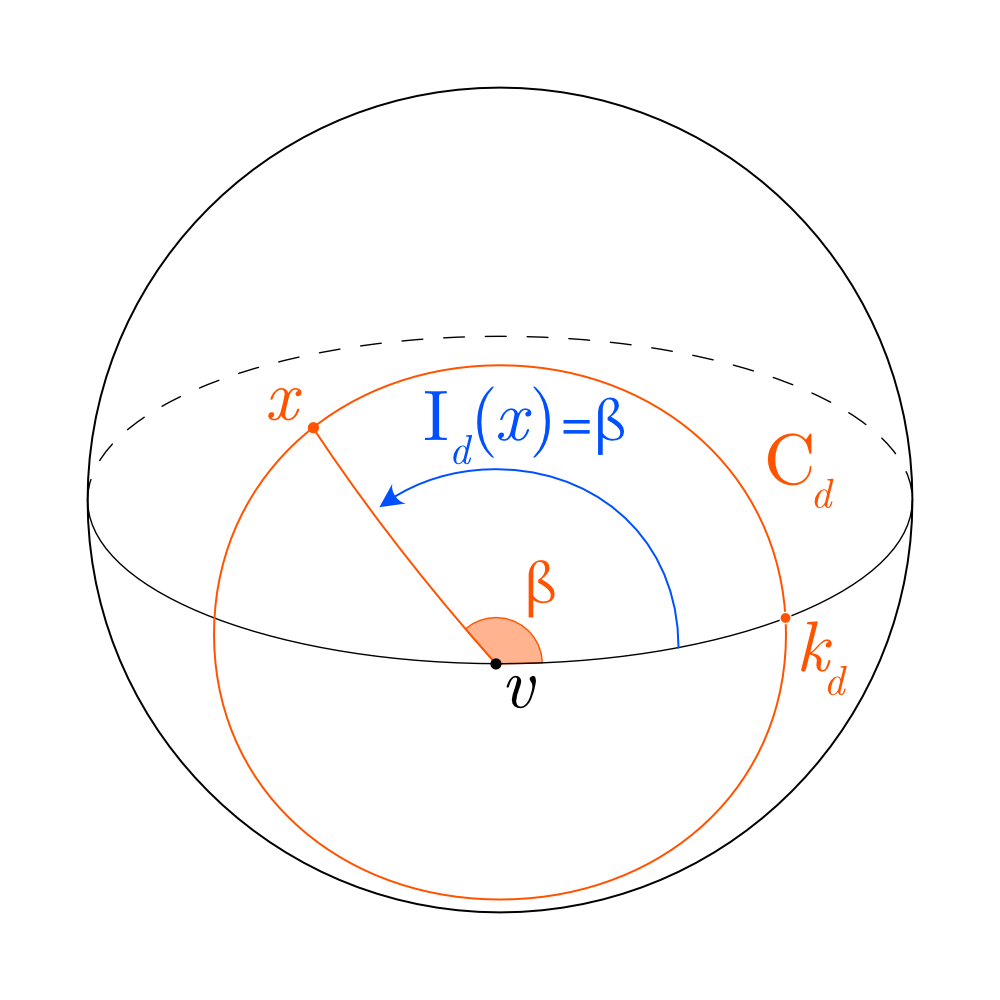}
\caption{Computing the longitude $I_d(x)$.}
\label{fig:Construction of Id}
\end{figure}

The set of functions $\{\rho_v^\theta \,|\, \theta\in S^1\}$ is clearly closed under function composition and isomorphic as a group to $S^1$. Call this group $G$. The group $G$ acts on $C_d$ and $C_{t(d)}$ by rotation, and one can easily check that the form of the action required in Corollary~\ref{cor:ci} holds. Furthermore, orienting $C_d$ and $C_{t(d)}$ by the right hand rule around $v$, the isomorphisms $I_d$ and $I_t$ preserve orientation, as required.

We wish to prove that $j_{v;d}$ is $G$-equivariant, that is, for all rotations $\rho_v^\theta\in G$ and $x\in C_d$, \[j_{v;d}(\rho_v^\theta(x)) = \rho_v^\theta(j_{v;d}(x))\text{,}\tag{\textasteriskcentered}\] which by Corollary~\ref{cor:ci} will complete the proof. 

Using the definitions, the left-hand side of $(\textasteriskcentered)$ gives \[j_{v;d}(\rho_v^\theta(x)) = \rho^\delta_{\rho_v^\theta(x)}(v)\text{.}\] Now by the $SO(3)$ conjugation rule, \[\rho^\delta_{\rho_v^\theta(x)}(v) = \rho_v^\theta(\rho^\delta_x(\rho_v^{-\theta}(v))) = \rho_v^\theta(\rho^\delta_x(v))\text{,}\] where $\rho_v^{-\theta}(v) = v$ since rotation around $v$ fixes $v$.

Focusing on the right-hand side of $(\textasteriskcentered)$, we see immediately that \[\rho_v^\theta(j_{v;d}(x)) = \rho_v^\theta(\rho^\delta_{x}(v))\] and thus $j_{v;d}$ is $G$-equivariant, as desired.
\end{proof}

We can now conclude that $j_{v;d}$ is an orientation-preserving bijective isometry from $C_d$ to $C_{t(d)}$, but what does this actually mean? First, consider the function $I_t\circ j_{v;d}\circ I_d^{-1}\colon S^1\to S^1$. By composition, since the identifications are orientation-preserving isometries, this map is an orientation-preserving isometry, and since it's a function from $S^1$ to $S^1$, it is therefore a rotation. Let this be rotation by some $\tau_{v;d}\in S^1$, so that $I_t\circ j_{v;d}\circ I_d^{-1}(x)$ = $\tau_{v;d}\cdot x$ for all $x\in S^1$.

This means that $j_{v;d}$ can be given by $I_t^{-1}(\tau_{v;d}\cdot I_d(x))$. Informally, this takes $x\in C_d$, rotates it by $\tau_{v;d}$ on $C_d$, and then scales it down to $C_{t(d)}$. We can use spherical geometry again to find $\tau_{v;d}$, which only depends on the distance $d$.

To compute $\tau_{v;d}$, take the equality \[I_t\circ j_{v;d}\circ I_d^{-1}(x) = \tau_{v;d}\cdot x\] and plug in $x = 0$ to get \[\tau_{v;d} = I_t(j_{v;d}(I_d^{-1}(0))) = I_t(j_{v;d}(k_d)) = I_t(\rho_{k_d}^\delta(v))\] We now need to find $I_t(\rho_{k_d}^\delta(v))$. To do this, we'll apply some spherical trigonometry to the triangle $\triangle(k_d, v, \rho_{k_d}^\delta(v))$. For the following, see Figure~\ref{fig:Independence of tau(d)}.

\begin{figure}[h]
\centering
\includegraphics[width = \linewidth]{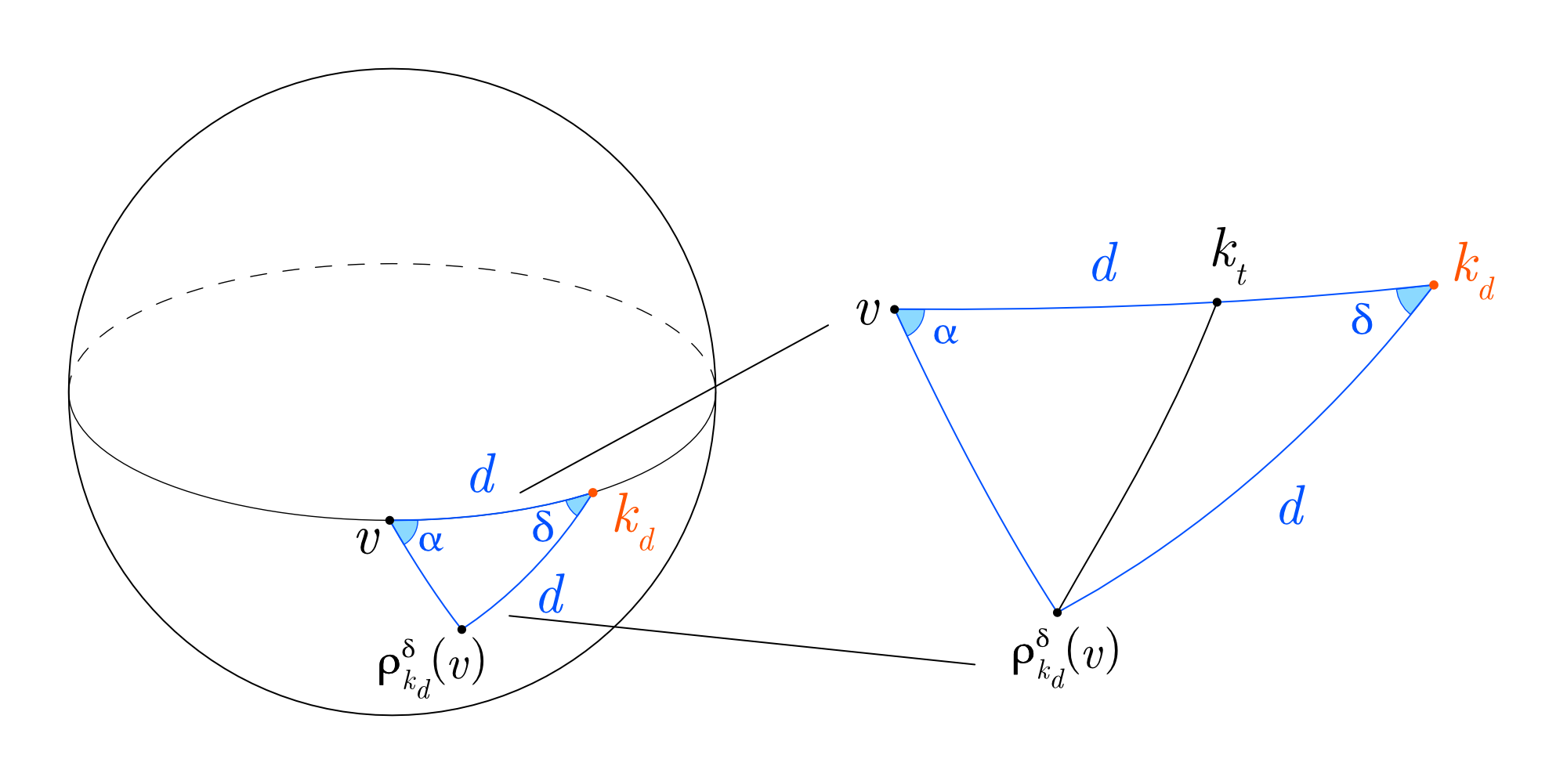}
\caption{The triangle $\triangle(k_d, v, \rho_{k_d}^\delta(v))$ on the sphere, showing the relevant angle to compute $I_t(\rho_{k_d}^\delta(v))$.}
\label{fig:Independence of tau(d)}
\end{figure}

By definition, $I_t(\rho_{k_d}^\delta(v))$ is the angle on $C_{t(d)}$ counterclockwise from the meridian $M_v$ to $\rho_{k_d}^\delta(v)$. Figure~\ref{fig:Independence of tau(d)} shows that the angle $\alpha$ is the angle from the meridian to $\rho_{k_d}^\delta(v)$ \textit{clockwise}, and so we get that $\tau_{v;d} = I_t(\rho_{k_d}^\delta(v)) = -\alpha$. By spherical side-angle-side we can see that this angle $\alpha$ depends only on the parameters $d$ and $\delta$ of triangle $\triangle(k_d, v, \rho_{k_d}^\delta(v))$. For $0\leq d\leq 2\pi$, let $\tau(d)$ be the angle $-\alpha$, where $\alpha$ is the other angle in the spherical triangle given by two edges of length $d$ connecting at the angle $\delta$, as in Figure~\ref{fig:Independence of tau(d)}.

\subsection{Reconstructing $J_v$}

Now we have an explicit form for $j_{v;d} = J_v|_{C_d}$ for all $0\leq d\leq2\pi$, and we can glue these explicit forms together to get the whole function. We will represent this in coordinates: Let $p\in S_v$, with $p$ not $v$ or its antipode. Write $p$ as $(d, \beta)$, where $d$ is the distance from $p$ to $v$ and $\beta$ is the ``longitude'' of $p$, i.e. $I_d(p)$. The image $J_v(p) = J_v(d, \beta)$ is the same as the image of $j_{v;d}(p)$, which is $(t(d), \tau(d) \cdot \beta)$. 

This explicit description of $J_v$ is enough to sketch $J_v^{-1}(T_v) = A_v^\delta$ for explicit examples. Figure~\ref{fig:Sketch of Allowable Set for Square} shows a sketch of $A_v^\delta$ for a vertex of the square embedded with its center at the origin using a fairly large choice of $\delta$. We encourage the reader to compute the shape of this allowable set for themselves, or at least to verify that axes in the shaded region actually rotate $v$ into $T_v$. We drew this by combining intuition about how $v$ moves under small rotations with simple heuristic estimates for $t(d)$ and $\tau(d)$. Such heuristic estimates can be quite accurate when $\delta$ is very small.
\begin{figure}[h!]
\centering
\includegraphics[width = 0.5\linewidth]{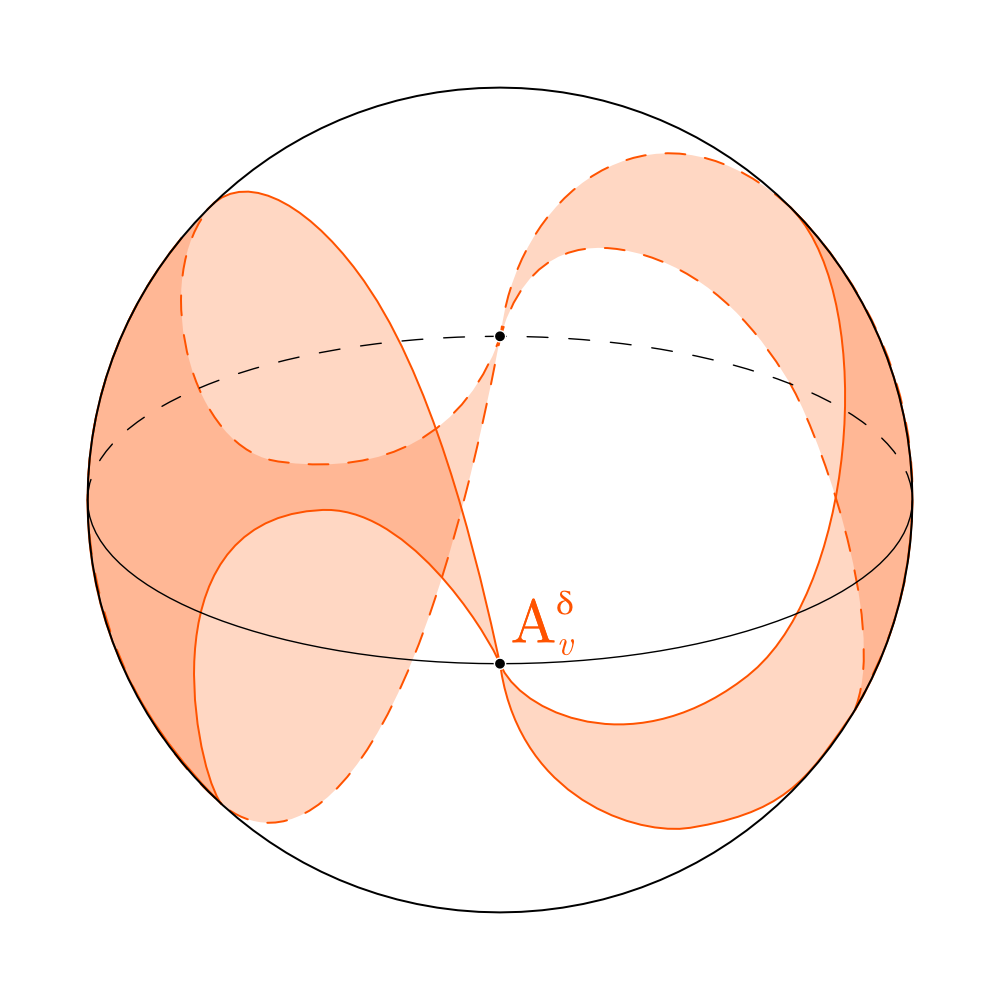}
\caption{A sketch of $J_v^{-1}(T_v) = A_v^\delta$, where $v$ is a vertex of the unit square centered at the origin.}
\label{fig:Sketch of Allowable Set for Square}
\end{figure}

The functions $t(d)$ and $\tau(d)$ are fairly complicated trigonometric functions as an artifact of spherical trigonometry, but analyzing this form can give us information about $J_v$ without needing to delve too deeply into the details. The next two lemmas capture the intuitive idea that, for different vertices $v$ and $w$ of $P$ with the same distance to the origin, the functions $J_v$ and $J_w$ ``do the same thing'' relative to their respective vertices. This idea is captured by noting that, since $v$ and $w$ are both on the equator of $S_v = S_w$, there is some rotation $\phi$ around the positive pole that takes $v$ to $w$. We first prove Lemma~\ref{lem:ppc}: for a point $p$, the coordinates of $p$ relative to $v$ are the same as the coordinates of $\phi(p)$ relative to $\phi(v) = w$. Since $J_v$ is defined by these coordinates, we can then show Lemma~\ref{lem:jsfr}, which captures the idea that $J_v$ and $J_{\phi(v)}$ are the same function relative to their respective vertices by expressing that idea as a conjugation by $\phi$.

\begin{lemma}\label{lem:ppc}
For a point $p\in S_v$, with $p$ not $v$ or its antipode, let $(d, \beta)_v$ be the coordinates of $p$ relative to $v$. Let $\phi$ be any rotation around the positive pole of $S_v$, i.e. around the $z$-axis. \[\phi((d, \beta)_v) = (d,\beta)_{\phi(v)}\]
\end{lemma}

\begin{figure}[h!]
\centering
\includegraphics[width = 0.5\linewidth]{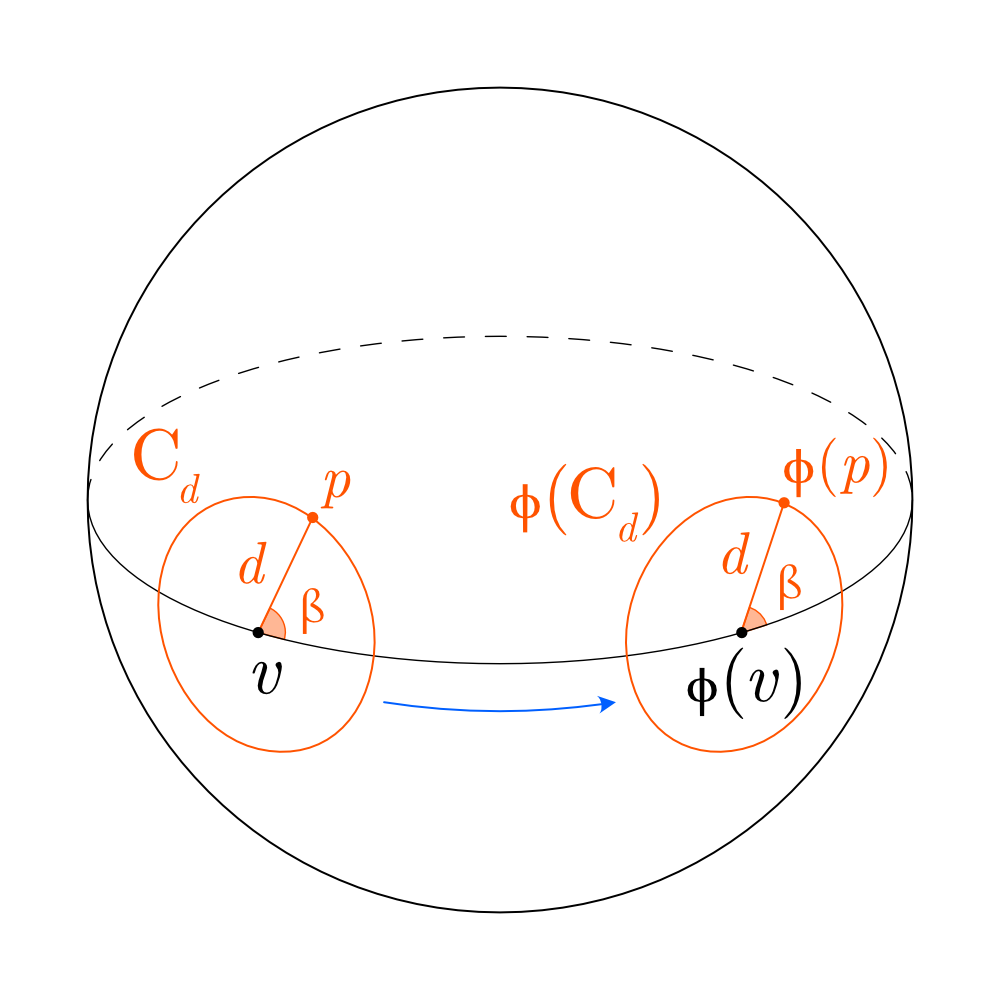}
\caption{The arrangement of $v$, $C_d$, and $p$ are rotated about the positive pole.}
\label{fig:Phi Commutes with Coordinatization}
\end{figure}

\begin{proof}
See Figure~\ref{fig:Phi Commutes with Coordinatization}. The value $d$ is the distance from $p$ to $v$, which is equal to the distance from $\phi(p)$ to $\phi(v)$ since $\phi$ is an isometry. Let $C_d$ be the circle of radius $d$ around $v$, and let $k_d$ be the intersection of $C_d$ with the meridian $M_v$. The longitude $\beta$ of $p$ relative to $v$ is the angle $\angle(k_d, v, p)$. The circle $\phi(C_d)$ is the circle of radius $d$ around $\phi(v)$. The meridian of $\phi(v)$ is $\phi(M_v)$, and the intersection of $\phi(C_d)$ and the meridian $\phi(M_v)$ is $\phi(k_d)$, so the longitude of $p$ relative to $\phi(v)$ is $\angle(\phi(k_d), \phi(v), \phi(p))$, which equals $\beta$ because $\phi$ preserves the surface geometry of $S_v$.
\end{proof}

\begin{lemma}\label{lem:jsfr}
Let $X\subseteq S_v$ be any set. Let $\phi$ be any rotation around the positive pole of $S_v$. \[J_v^{-1}(X) = \phi^{-1}(J_{\phi(v)}^{-1}(\phi(X)))\]
\end{lemma} 

\begin{proof}
We'll begin by proving that, for all $0\leq d\leq2\pi$ and for all $x\in C_{t(d)}$, 
\[j^{-1}_{v;d}(x) = \phi^{-1}(j^{-1}_{\phi(v);d}(\phi(x)))\]
If $d = 0$ or $d=2\pi$, then $t(d) = 0$ and hence $x = v$, in which case the equality is obvious. Consider now the case for $0<d<2\pi$. 

Since $0<d<2\pi$ and $x\in C_{t(d)}$, $x$ is neither $v$ nor its antipode. This means that we can write $x = (t(d), \tau(d) + \beta)_v$. We can now use the co-ordinate representation of $j^{-1}_{v;d}$ and apply Lemma~\ref{lem:ppc} to see \[j^{-1}_{v;d}(x) = j^{-1}_{v;d}((t(d), \tau(d) \cdot \beta)_v) = (d, \beta)_v\] and on the right we get \[\phi^{-1}(j^{-1}_{\phi(v);d}(\phi(x))) = \phi^{-1}(j^{-1}_{\phi(v);d}((t(d), \tau(d) \cdot \beta)_{\phi(v)}))\] \[=\phi^{-1}((d, \beta)_{\phi(v)}) = (d,\beta)_v\] which shows the desired equality.

To show that $J_v^{-1}(X) = \phi^{-1}(J_{\phi(v)}^{-1}(\phi(X)))$, we first show the inclusion $J_v^{-1}(X) \subseteq \phi^{-1}(J_{\phi(v)}^{-1}(\phi(X)))$. Let $y\in J_v^{-1}(X)$. This means that $J_v(y) \in X$, so by letting $d$ be the distance from $y$ to $v$, $j_{v;d}(y)\in X$.

We now apply the above equality to show that \[j^{-1}_{v;d}(j_{v;d}(y)) = \phi^{-1}(j^{-1}_{\phi(v);d}(\phi(j_{v;d}(y))))\] On the left, we just get $y$. On the right, we see that $\phi(j_{v;d}(y))\in \phi(X)$, so the point $j^{-1}_{\phi(v);d}(\phi(j_{v;d}(y)))$ gets mapped into $\phi(X)$ by $j_{\phi(v);d}$ and is therefore in the preimage $J_{\phi(v)}^{-1}(\phi(X))$. This means that $\phi^{-1}(j^{-1}_{\phi(v);d}(\phi(j_{v;d}(y))))\in \phi^{-1}(J_{\phi(v)}^{-1}(\phi(X)))$, but by the above equality, the point on the left is just $y$, and thus $y\in \phi^{-1}(J_{\phi(v)}^{-1}(\phi(X)))$ as desired.

The other inclusion, $\phi^{-1}(J_{\phi(v)}^{-1}(\phi(X))) \subseteq J_v^{-1}(X)$, follows from the above inclusion by noting that $\phi^{-1}$ is also a rotation about the positive pole. Let $Y = \phi(X)$, so $X = \phi^{-1}(Y)$, which follows since $\phi^{-1}$ is a bijection. Let $w = \phi(v)$. By our inclusion, with rotation $\phi^{-1}$, point $w$ and set $Y$, \[J_w^{-1}(Y) \subseteq \phi(J_{\phi^{-1}(w)}^{-1}(\phi^{-1}(Y)))\] On the right, $\phi(J_{\phi^{-1}(w)}^{-1}(\phi^{-1}(Y))) = \phi(J_{v}^{-1}(X))$ by our definitions, and on the left we see $J_w^{-1}(Y) = J_{\phi(v)}^{-1}(\phi(X))$. This new inclusion is \[J_{\phi(v)}^{-1}(\phi(X))\subseteq \phi(J_{v}^{-1}(X))\] and applying $\phi^{-1}$ to both sides of this new inclusion gives the desired result.
\end{proof}

We now apply the development of $J_v$ to prove our main  theorems.

\section{Our Main Theorems}

\subsection{Double-Arch Polygonal Sections}

\begin{definition}[Double-Arch Polygon]
Let $v, v'\in \R^2$ be on the $x$-axis. Let $\ell, \ell'$ be perpendicular to the $x$-axis so that $\ell$ passes through $v$ and $\ell'$ passes through $v'$. A \textit{polygonal arch} is a polygon in $\R^2$ constructed as the convex hull of $v$ and $v'$, as well as a finite set of other vertices $W$, where each $w\in W$ is strictly above the $x$-axis and strictly between the two lines $\ell$ and $\ell'$. See Figure~\ref{fig:Polygonal Arch Definition} for an example.

\begin{figure}[h!]
\centering
\includegraphics[width = 0.5\linewidth]{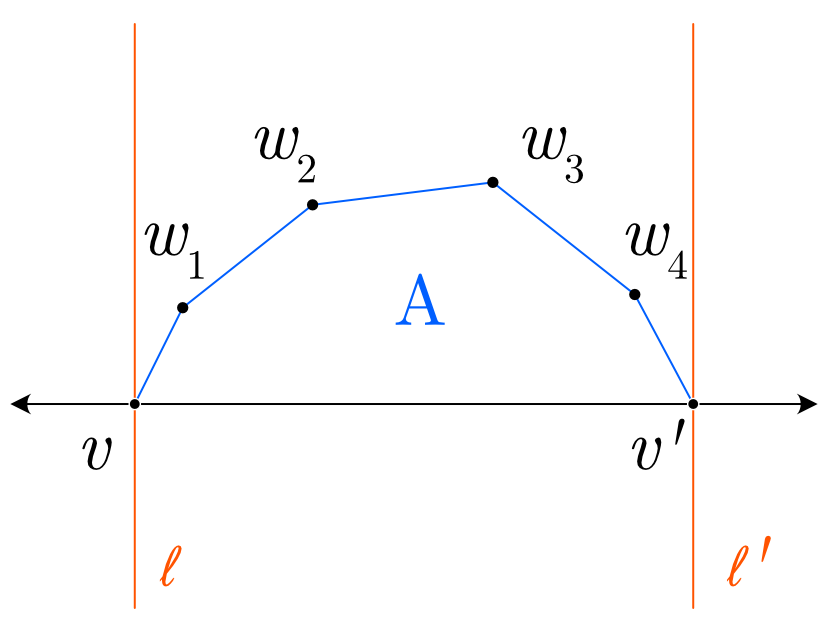}
\caption{An example of a polygonal arch, $A$.}
\label{fig:Polygonal Arch Definition}
\end{figure}

A trivial arch is an arch where the set $W$ is empty. A \textit{double-arch polygon} $P$ is constructed from two arches $A$ and $A'$ which share endpoints $v$ and $v'$ by flipping $A'$ over the $x$-axis and gluing the flipped arch to $A$ at the shared vertices $v$ and $v'$. A double-arch polygon is \textit{nontrivial} if it is constructed from two nontrivial arches $A$ and $A'$.

\end{definition}

Many polygons are nontrivial double-arch, including all the regular polygons besides the triangle; see Figure~\ref{fig:Double-Arch Polygon Examples} for the general pattern for regular polygons.

\begin{figure}[h!]
\centering
\includegraphics[width = 0.6\linewidth]{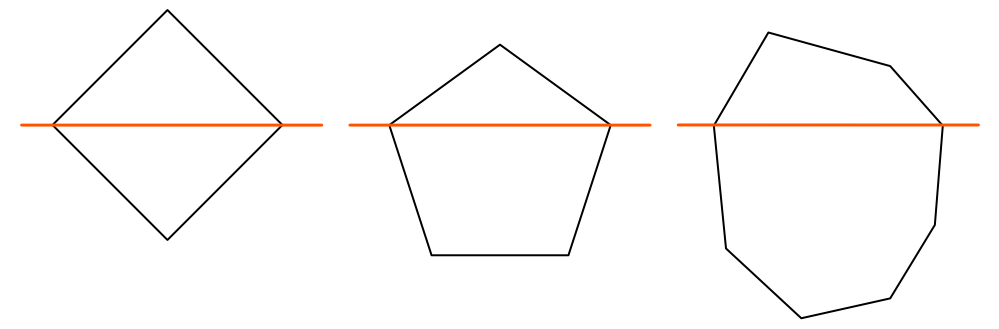}
\caption{The square, regular pentagon, and a less regular polygon are shown to be nontrivial double-arch.}
\label{fig:Double-Arch Polygon Examples}
\end{figure}

Not all polygons are nontrivial double-arch, though, as is shown in Figure~\ref{fig:Double-Arch Nonexample}. The example in this figure is constructed by taking a right-angled triangle, flipping it over the perpendicular bisector to its hypotenuse, and taking the convex hull of the three original vertices and the three vertices post-flip.

\begin{figure}[h!]
\centering
\includegraphics[width = 0.45\linewidth]{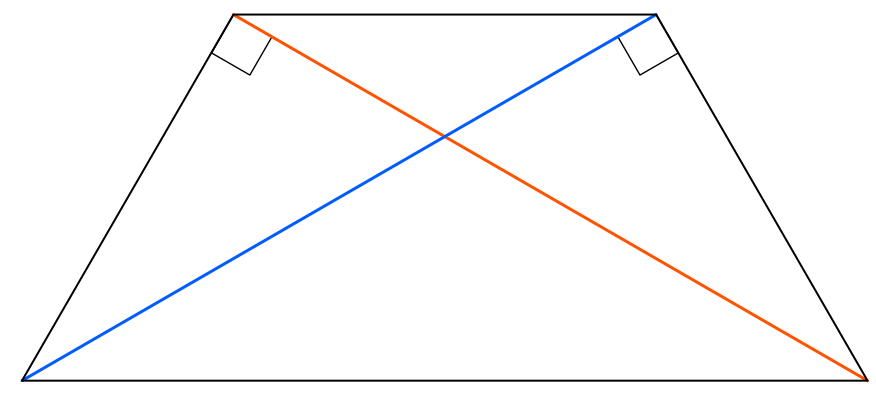}
\caption{A polygon which is not nontrivial double-arch.}
\label{fig:Double-Arch Nonexample}
\end{figure}

\hfill
\begin{theorem}\label{thm:dap}
If a flat polygon $P$ is nontrivial double-arch, then it is locally Rupert.
\end{theorem}

\begin{proof}
Let $r$ be the segment of the $x$-axis between the endpoints $v$ and $v'$ of $P$. Since $P$ is nontrivial, the line segment $r$ is in the interior of $P$ at all points other than $v$ and $v'$. Translate $P$ so that the origin bisects $r$. Let $S_v$ be the unique sphere in $\R^3$ centered at the origin and passing through $v$. Since we translated $P$ so that the origin is the bisector of $r$, the distance from the origin to $v$ is the same as the distance from the origin to $v'$, and thus $v'\in S_v$. Furthermore, by construction $v$ and $v'$ are antipodes on $S_v$.

For the following, see Figure~\ref{fig:v is Allowable for w}. For every vertex $w$ of $P$ other than $v$ and $v'$, and for sufficiently small $\delta>0$, $A_w^\delta$ contains $v$. This can be seen by constructing the unique line $m$ through $w$ perpendicular to to $r$, and letting the intersection of $m$ and $r$ be $x$. The point $x$ is inside $P$, as is $w$, so the straight line between them is contained within $P$ by convexity; furthermore, all the points of $m$ besides $w$ are in the interior of $P$. For sufficiently small $\delta>0$, $\pi(\rho_v^\delta\cdot w)\in m$ and since $\rho_v^\delta$ isn't the identity, $\pi(\rho_v^\delta\cdot w)\not = w$, so $\pi(\rho_v^\delta\cdot w)\in\intr(P)$.

\begin{figure}[h!]
\centering
\includegraphics[width = 0.5\linewidth]{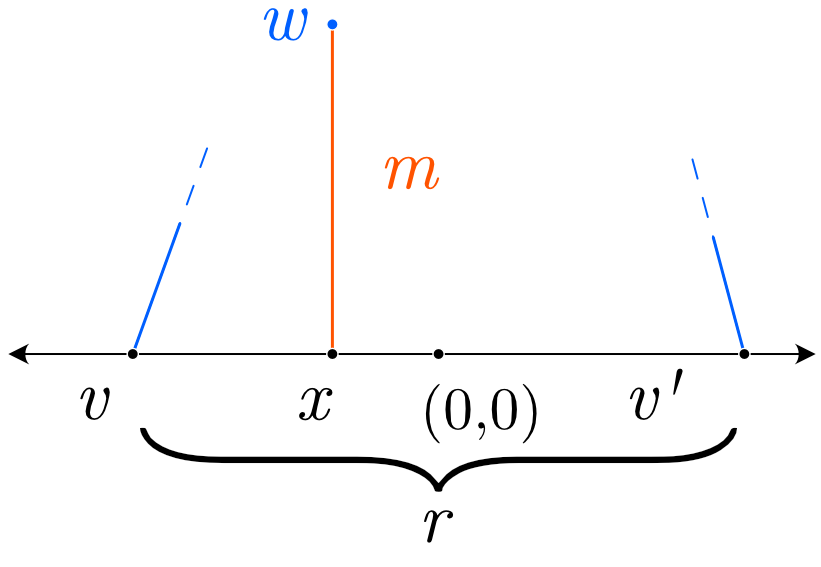}
\caption{A vertex $w$ of a double-arch polygon and the line $m$ which $\pi(w)$ travels along when rotating about the axis pointing at $v$.}
\label{fig:v is Allowable for w}
\end{figure}

Let $\mathcal A^\delta = \bigcap\limits_{w\in P\setminus\{v, v'\}} A_w^\delta$, which is open since the indexing set is finite and each component set is open. Since $v$ is in each component set of this intersection, $v\in \mathcal A^\delta$. If $A_v^\delta$ intersects $A_{v'}^\delta$ in $\mathcal A^\delta$, say at an axis $a\in \mathcal A^\delta$, then $a$ is in every allowable set simultaneously and thus $\rho_a^\delta$ is a Rupert rotation. If we can produce such an axis $a$ for all $\delta$ small enough, that shows that $P$ is locally Rupert.

We claim that such an intersection happens at an axis which sends $\pi(v)$ along $r$.  Let $R$ be the intersection of $S_v$ with the preimage $\pi^{-1}(r)$. See Figure~\ref{fig:R as the Preimage of r}. $R$ is a great circle by construction, and is contained within $T_v$ and $T_{v'}$ at all points other than $v$ and $v'$, since $r$ is in the interior of $P$ everywhere but $v$ and $v'$. 

\begin{figure}[h!]
\centering
\includegraphics[width = 0.5\linewidth]{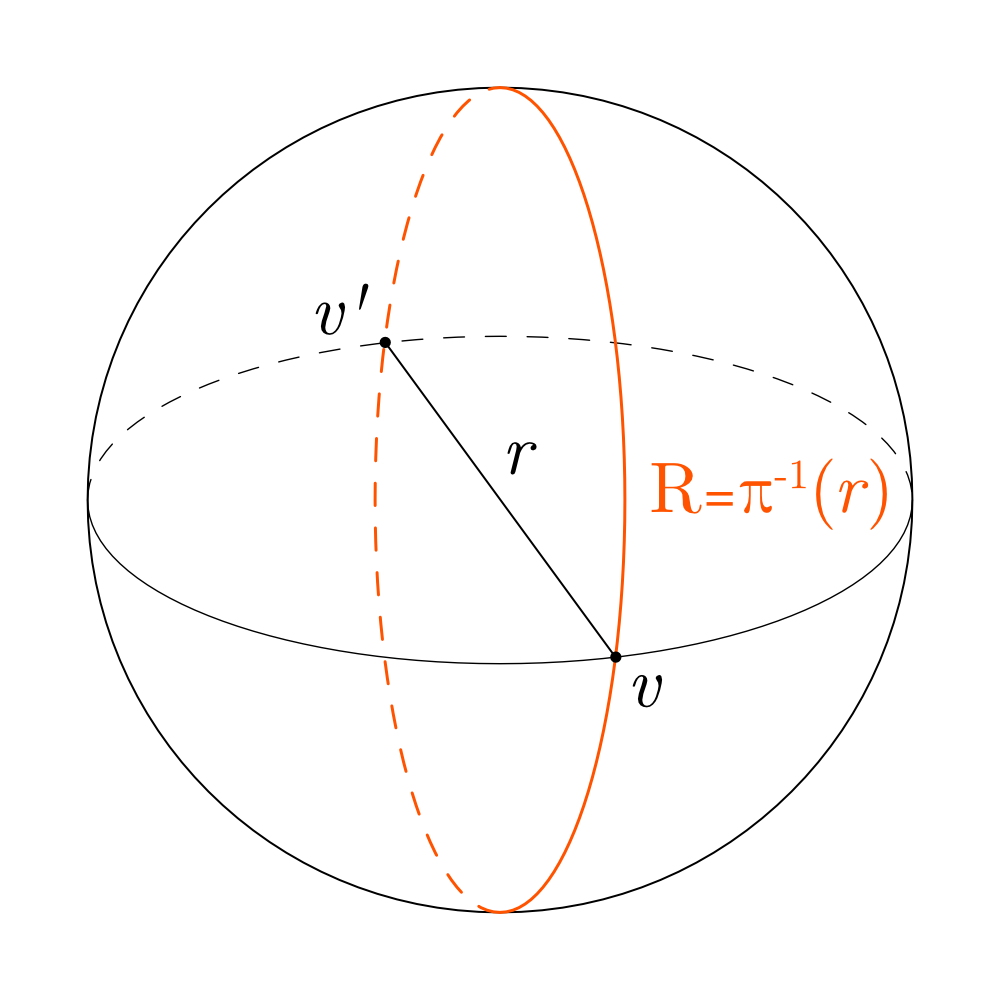}
\caption{The line $R$, shown as the preimage on the sphere of $r$.}
\label{fig:R as the Preimage of r}
\end{figure}

Let $\phi$ be the rotation of $S_v$ about its positive pole taking $v$ to $v'$ (which in this case will be rotation by $\pi$). Let $\varsigma\colon S_v\to S_v$ be the antipodal map. Clearly $\phi(R) = R$ and $\varsigma(R) = R$, as $R$ is a great circle orthogonal to the equator. 

Let $F = J_v^{-1}(R)$ and $F' = J_{v'}^{-1}(R)$. We claim that $F = F'$.

We first claim that $\phi(F) = F$. We'll prove this by showing that $\phi(F)\subseteq F$. Since $\phi$ is an involution, this also shows that $F\subseteq\phi(F)$ and completes our equality. An axis $f\in F$ if and only if $\rho_f^\delta(v)\in R$. Let $f\in F$. We claim that $\phi(f)\in F$, i.e. that $\rho_{\phi(f)}^\delta(v)\in R$. We can write $\rho_{\phi(f)}^\delta(v)$ as a conjugation by the SO(3) conjugation rule, Theorem~\ref{thm:scr}: \[\rho_{\phi(f)}^\delta(v) = \phi(\rho_f^\delta(\phi^{-1}(v))) = \phi(\rho_f^\delta(v'))\] Rotations preserve antipodes, and so since $\rho_f^\delta$ takes $v$ onto $R$, it must take the antipode $v'$ onto the antipode $\varsigma(R)$, which by the above is just $R$. Thus $\rho_f^\delta(v')\in R$, and since $\phi(R) = R$, therefore $\rho_{\phi(f)}^\delta(v) = \phi(\rho_f^\delta(v'))\in R$ as desired.

By Lemma 6, $J_v^{-1}(R) = \phi^{-1}(J_{\phi(v)}^{-1}(\phi(R)))$. On the left we get $F$. On the right, since $\phi(v) = v'$ and $\phi(R) = R$, we get $\phi^{-1}(F')$. Thus $F = \phi^{-1}(F')$, and by the above, $\phi(F) = F$, so applying $\phi$ we see $F = F'$ as desired.

As we've covered, $R$ is contained within the regions $T_{v}$ and $T_{v'}$. By definition, then, $F\subseteq A_{v}^\delta$, since $F$ is the set of axes taking $v$ into $R\subseteq T_v$, and similarly $F'\subseteq A_{v'}^\delta$. Furthermore, the points of $F$ are $j^{-1}_{v;d}(x)$ for any $d$ and a point $x \in C_{t(d)}\cap R$. This intersection is never empty, since $R$ passes through $v$, and so for all $d$ there is a point of $F$ on $C_d$. 

Since $F = F'$, $F\subseteq A_{v}^\delta\cap A_{v'}^\delta$. Hence $A_v^\delta$ intersects $A_{v'}^\delta$ at $F = F'$ on every circle $C_d$ centered at $v$. Since $v\in \mathcal A^\delta$ and $\mathcal A^\delta$ is open, by taking small enough $d$, we get that $C_d\subseteq \mathcal A^\delta$. Hence, for this $d$, $A_v^\delta$ intersects $A_{v'}^\delta$ on $F\cap C_d\subseteq\mathcal A^\delta$ for all sufficiently small $\delta>0$, showing that $P$ is locally Rupert.
\end{proof}

From here, we can use Lemma~\ref{lem:psbs} to prove the corollary

\begin{theorem*}[A]
If a polyhedron $Q$ has a double-arch polygonal section $P$, then $Q$ is locally Rupert.\qed
\end{theorem*}

\subsection{Prisms Over Polygons}

It is somewhat unsatisfying that Theorem~\nameref{thm:mainA} does not handle the case of the Cube, despite that polyhedron's historical importance. This dissatisfaction leads us to think about how to extend the above theorem. Above, we sought a Rupert rotation, a rotation $\rho$ so that $\pi(\rho(P))\subseteq \intr(\pi(P))$. This rotation, from the perspective of $\pi$, ``shrinks'' $P$. We could just as easily seek a rotation $\sigma$ so that $\pi(P)\subseteq \intr(\pi(\sigma(P)))$. This rotation would ``expand'' $P$. Call such a rotation \textit{reverse Rupert}. The definitions for a reverse Rupert polyhedron and a locally reverse Rupert polyhedron follow the analogous definitions for Rupert and locally Rupert.

By taking the local perspective, we can again simplify our problem. 

\begin{definition}[Prism over a Polygon]
Let $P$ be a polygon embedded in $\R^3$ on the $xy$-plane, i.e. a flat polygon. A \textit{prism over $P$}, say $R$, is a polyhedron obtained from $P$ by taking the set $\{(x, y, z)\in \R^3 | (x, y)\in P, -h \leq z \leq h\}$ for some $h>0$, which functionally just extends $P$ vertically by $h$ above and below the $xy$-plane.
\end{definition}

See Figure~\ref{fig:Prism Section Definition} for the following definition.

\begin{figure}[h!]
\centering
\includegraphics[width = 0.5\linewidth]{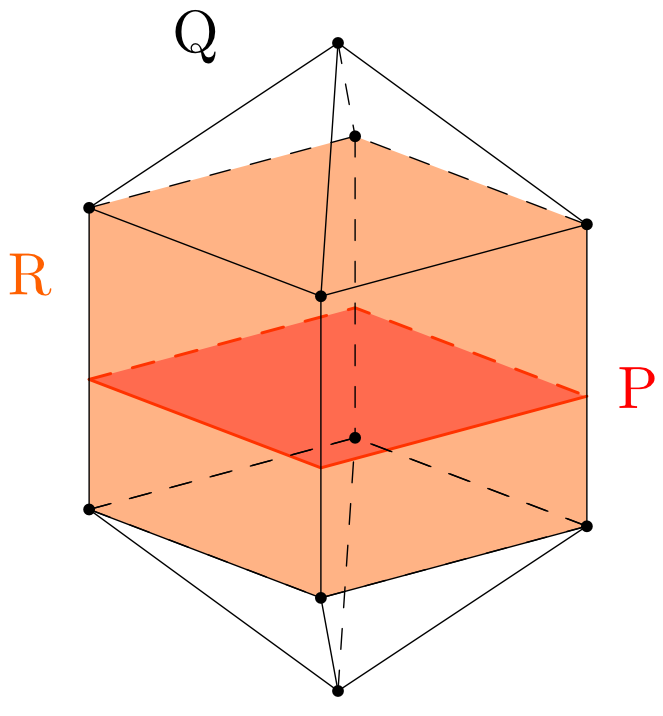}
\caption{An example of a prism section $R$ over $P$, where $P$ is the square.}
\label{fig:Prism Section Definition}
\end{figure}

\begin{definition}[Prism Section]
Let $L_h = \{(x, y, z)\in \R^3 | -h \leq z \leq h\}$. Let $Q$ be a polyhedron. A \textit{prism section} of $Q$ is a prism $R$ over a polygon $P$ so that $R$ is the intersection of $Q$ and $L_h$ for some $h>0$.
\end{definition}

In analogy with polygonal sections, we see that $\pi(Q) = \pi(R) = P$, since every vertex of $Q$ must lie in $\pi^{-1}(\pi(R))$: if a vertex were to lie outside of $\pi^{-1}(\pi(R))$, by taking the convex hull of that vertex and $R$ we would see that $Q$ intersects $L_h$ in more than just $R$.

To complete the analogy to polygonal sections, we need a prism section version of Lemma~\ref{lem:psbs}, which we now provide.

\begin{lemma}[Prism Section Bootstrap]\label{lem:prsbs}
Let $Q$ be a polyhedron with a prism section $R$. If $R$ is locally reverse Rupert, so is $Q$.
\end{lemma}

\begin{proof}
Let $\eps>0$ be given. We wish to show that there is a rotation $\sigma \in SO(3)$ with $|\sigma|<\eps$ so that $\pi(Q)\subseteq \intr(\pi(\sigma(Q)))$. By the locally reverse Rupert property of $R$, there exists $\sigma\in SO(3)$ with $|\sigma|<\eps$ so that $\pi(R)\subseteq \intr(\pi(\sigma(R)))$. Since $R\subseteq Q$, $\intr(\pi(\sigma(R))) \subseteq \intr(\pi(\sigma(Q)))$ and hence \[\pi(Q) = \pi(R) \subseteq \intr(\pi(\sigma(R))) \subseteq \intr(\pi(\sigma(Q)))\text{.}\] 
\end{proof}

We now provide a result which extends Theorem~\ref{thm:dap}.

\begin{theorem}\label{thm:prp}
Let $R$ be a prism over $P$. If $P$ is locally Rupert, then $R$ is locally reverse Rupert.
\end{theorem}

\begin{proof}
In order to show that a rotation $\sigma$ is reverse Rupert, we need to show that $\pi(R)\subseteq \intr(\pi(\sigma(R)))$. The set on the left, $\pi(R)$, is just $P$ by construction. The set on the right is convex since it is the interior of a projection of a convex set $\sigma(R)$, so to show that $P\subseteq \intr(\pi(\sigma(R)))$, it suffices to show that each vertex $v\in P$ is contained within $\intr(\pi(\sigma(R)))$. Let $\eps>0$ be given. We will construct a reverse Rupert rotation $\sigma$ with $|\sigma|<\eps$.

Let $v$ be a vertex of $P$, and let $S_v$ be the unique sphere centered at the origin and passing through $v$. Consider the intersection $S_v\cap R$, shown in Figure~\ref{fig:Prism Intersects Tangent Sphere}. This is a somewhat complex shape.
\begin{figure}[h!]
\centering
\includegraphics[width = 0.5\linewidth]{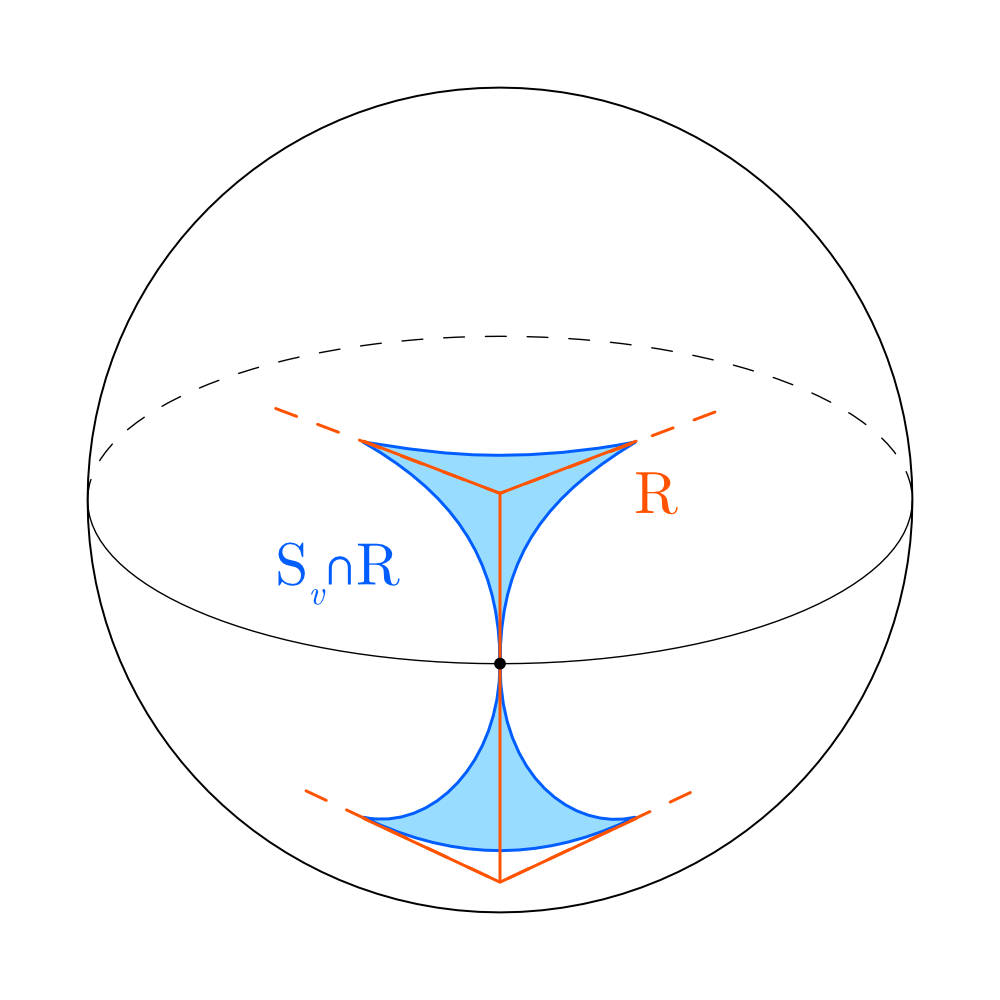}
\caption{The intersection $S_v\cap R$.}
\label{fig:Prism Intersects Tangent Sphere}
\end{figure}

Notice, though, that $v$ lies on a vertical edge of $R$, say $E$. Since every point of $E$ is further from the origin than $v$, $S_v$ is tangent to $E$ with point of tangency $v$. Thus, since $h>0$, near $v$, the intersection $S_v\cap \intr(R)$ looks like the intersection $S_v \cap (\pi^{-1}(\intr(P))) = T_v$, an example of which was shown in Figure~\ref{fig:Tv}.

In more careful terms, we know that 
\begin{enumerate} 
\item $\intr(R)\subseteq \pi^{-1}(\intr(P))$, and
\item There is a small enough disc on $S_v$ centered at $v$, say $D_v$, so that every point $(x,y,z)\in D_v$ has $|z|<h$.
\end{enumerate}

Every point $(x,y,z)\in (D_v\cap \pi^{-1}(\intr(P)))$ has $|z|<h$, so since $(x,y,z)\in \pi^{-1}(\intr(P))$ with $|z|<h$, $(x,y,z)\in \intr(R)$. Hence, $(D_v\cap \pi^{-1}(\intr(P)))\subseteq \intr(R)$. 

Furthermore, $D_v\cap \pi^{-1}(\intr(P)) = D_v\cap \pi^{-1}(\intr(P)) \cap S_v$ since $D_v\subseteq S_v$, but since $\pi^{-1}(\intr(P)) \cap S_v = T_v$, $D_v\cap \pi^{-1}(\intr(P)) = D_v\cap T_v$. Combining these two facts we get \[D_v\cap T_v\subseteq \intr(R)\text{.}\]

By restricting $\gamma>0$ small enough, for any rotation $\rho$ with $|\rho|<\gamma$, $\rho\cdot v$ can be guaranteed to stay in a small disc centered at $v$ on $S_v$. Let $\gamma_v$ be given small enough to keep $\rho\cdot v$ inside $D_v$. Let $\alpha$ be the minimum over $\eps$ and $\gamma_w$ for all vertices $w\in P$.

By the locally Rupert property of $P$, there exists a rotation $\rho$ with $|\rho| = \delta$, $\delta<\alpha$ so that, for all vertices $v\in P$, $\pi(\rho\cdot v)\in \intr(\pi(P))$ and hence $\rho\cdot v\in T_v$. If we imagine keeping $R$ stationary and rotating $P$ by $\rho$, since $\delta< \gamma_v$ for all vertices $v\in P$, we know that $\rho\cdot v\in T_v\cap D_v\subseteq \intr(R)$ for all $v$. Taking the convex hull of these vertices on the left, we get that $\rho(P) \subseteq \intr(R)$. By applying the inverse of $\rho$, say $\sigma$, to both $\rho(P)$ and $\intr(R)$, we get that $P\subseteq \sigma(\intr(R))$. Furthermore, by applying $\pi$, since $\pi(P) =P$ we see that $P\subseteq \pi(\sigma(\intr(R)))$, but the interior operator commutes over $\sigma$ and $\pi$ and thus $P\subseteq \intr(\pi(\sigma(R)))$, as desired. Since $\sigma$ has $|\sigma| = |\rho| < \alpha < \eps$, this shows that $R$ is locally reverse Rupert.
\end{proof}

Again, we can apply our bootstrap, Lemma~\ref{lem:prsbs}, to show that if a polyhedron $Q$ has a prism section $R$ over $P$ and $P$ is locally Rupert, then $Q$ is locally reverse Rupert. Combining this with Theorem~\ref{thm:dap}, we get the corollary

\begin{theorem*}[B]
If a polyhedron $Q$ has a prism section $R$ over $P$ and $P$ is nontrivial double-arch, then $Q$ is locally reverse Rupert.\qed
\end{theorem*}

We now return to the usual convention that ``polyhedron,'' when used alone, means ``unoriented polyhedron.'' Theorems \nameref{thm:mainA} and \nameref{thm:mainB} give a pleasing symmetry in some important cases, since if we take the dual of a polyhedron with a polygonal section $P$, we'll get a prism section over the dual of $P$. Regular polygons are self-dual, so we get the ``duality'' corollary below.

\begin{corollary}\label{cor:dual}
A polyhedron $Q$ has a regular Rupert polygonal section that shows that $Q$ is locally Rupert by Theorem~\nameref{thm:mainA} if and only if the dual $D$ of $Q$ has a prism section over a Rupert regular polygon, which shows that $D$ is locally reverse Rupert by Theorem~\nameref{thm:mainB}.\qed
\end{corollary}

\section{Survey of Results}

Here we offer a survey of the important polyhedra which these theorems prove are locally Rupert or locally reverse Rupert. We do not resolve any unsolved polyhedra in any of the four most important classes --- Platonic, Archimedean, Catalan, and Johnson --- but we do recover many results of previous authors with regards to these polyhedra. These theorems, by virtue of their generality, handle a wide range of less-regular solids; for example, we expect that our theorems can recover much of Theorem 3 from \cite{Steininger2021} in regards to the Johnson solids. 

When trying to spot when Theorem~\nameref{thm:mainA} will be applicable, one should look for a ``seam'' around the polyhedron, then double-check that this seam is, in fact, nontrivial double-arch. Spotting when Theorem~\nameref{thm:mainB} is applicable is slightly more difficult, but one should look for a ring of faces around the polyhedron which all have normal vectors lying in the same plane.

Since polygonal and prism sections are often easy to spot, we will give examples of each on the Octahedron and Cube, then simply list the names of the other polyhedra which we cover out of the Platonic, Archimedean, and Catalan solids. Where possible, we list who first proved that a given polyhedron is Rupert, but we understand that there may be errors in this list. No disrespect is meant to any author by their omission. 

\subsection{Platonic Solids}

\begin{figure}[h!]
\centering
	\begin{subfigure}{.5\textwidth}
	\centering
	\includegraphics[height = .75\linewidth]{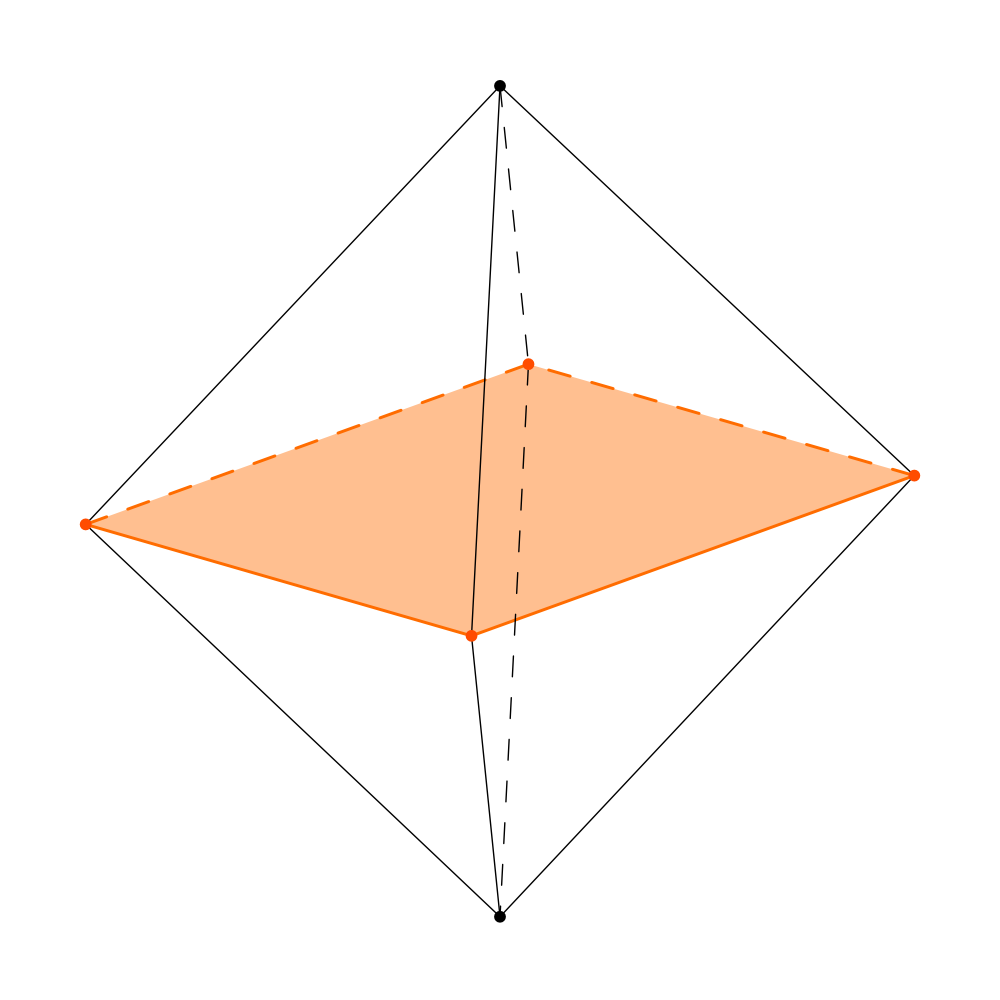}
	\caption{A polygonal section for the Octahedron.}
	\label{fig:Octahedron Polygonal Section 2}
	\end{subfigure}%
	\begin{subfigure}{.5\textwidth}
	\centering
	\includegraphics[height = .75\linewidth]{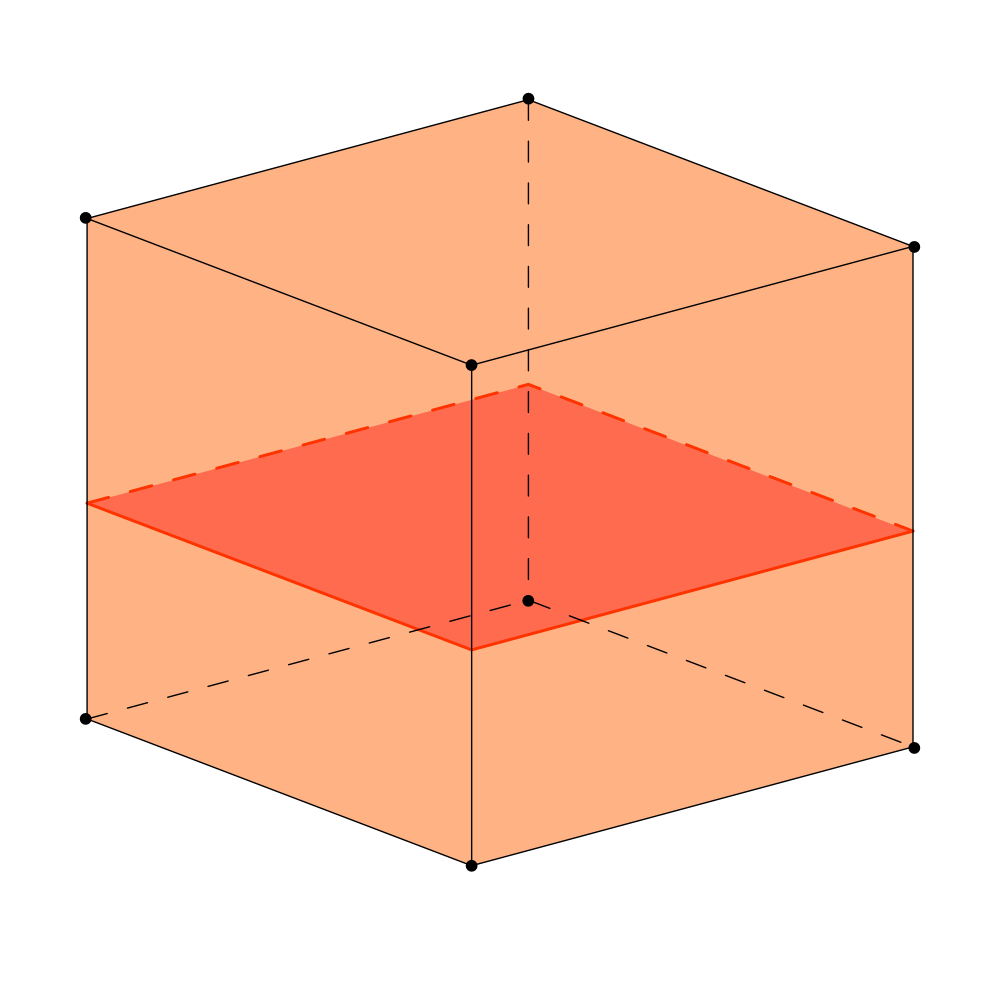}
	\caption{A prism section for the Cube.}
	\label{fig:Cube Prism Section 2}
	\end{subfigure}
\caption{}	
\label{fig:Platonic Sections 2}
\end{figure}
The Octahedron is proven locally Rupert by Theorem~\nameref{thm:mainA}, and the Cube is proven locally reverse Rupert by Theorem~\nameref{thm:mainB}. See Figure~\ref{fig:Platonic Sections 2} for the polygonal section and prism section, respectively. These results recover the classical case of the Cube, and the case of the Octahedron as covered in \cite{Scriba1968}. 

In regards to the other Platonic solids: the Tetrahedron has a polygonal section, an equilateral triangle, but that section is trivial double-arch. The Dodecahedron and Icosahedron have neither polygonal nor prism sections. A possible approach for these two, in the local case, is discussed in Subsection~\ref{subsec:apatr}.

\subsection{Archimedean Solids}
Theorem~\nameref{thm:mainA} proves that the Cuboctahedron and Icosidodecahedron are locally Rupert, and Theorem~\nameref{thm:mainB} proves that the Truncated Cube, Truncated Octahedron, Rhombicuboctahedron, Truncated Cuboctahedron, and Truncated Icosidodecahedron are locally reverse Rupert.\footnote{If one wishes to count the Elongated Square Gyrobicupola of Grunbaum as an Archimedean solid, Theorem~\nameref{thm:mainB} handles that case as well.} All of these results are recoveries of previous results: all of these besides the Truncated Icosidodecahedron were done in \cite{Chai2018}, and the Truncated Icosidodecahedron was done in \cite{Steininger2021}. We also cover the infinite class of prisms besides the triangular prism.

These theorems still do not handle the particularly tricky case of the Rhombicosidodecahedron, discussed also in \cite{Steininger2021}. Their results seem to indicate that if this polyhedron is Rupert, then it has very small Nieuwland constant. This might indicate that if it \textit{is} Rupert, then it might have a chance of being Rupert in a ``local'' sense. The fact that our result doesn't cover this polyhedron might then serve as meagre evidence in favor of those authors' conjecture that it is not Rupert.
\subsection{Catalan Solids}

Theorem~\nameref{thm:mainA} proves that the Triakis Octahedron, Triakis Hexahedron, Deltoidal Icositetrahedron, Disdyakis Dodecahedron, and Disdyakis Triacontahedron are locally Rupert. Theorem~\nameref{thm:mainB} proves that the Rhombic Dodecahedron\footnote{This can be very hard to see, but if one looks at the orthogonal projection onto the plane normal to an axis pointing at a vertex of valence three, the ring of faces we are looking for all get projected to straight lines. One can also look at the ``seam'' on the Cuboctahedron and see where the seam ends up after dualization.} and Rhombic Triacontahedron\footnote{Again, look at the orthogonal projection along a vertex of valence five, or trace the seam from the Icosidodecahedron through duality.} are locally reverse Rupert. All of these polyhedra were proven Rupert in \cite{Steininger2021}. We also cover the infinite class of bipyramids besides the triangular bipyramid.

\section{Further Work}

\subsection{Antiprisms and Trapezohedra}
\label{subsec:apatr}

The Archimidean solids contain two infinite classes: the prisms and the antiprisms. Similarly, the Catalan solids contain two infinite classes, duals to the previous two: the bipyramids and trapezohedra. One way to see our two theorems is that we prove that the bipyramids and prisms are locally Rupert and locally reverse Rupert, respectively, and then give ``bootstrap'' lemmas to say that if a polyhedron $Q$ has a section which looks sufficiently similar to a bipyramid or prism, then that polyhedron ``inherits'' from those infinite classes and is locally (reverse) Rupert as well.

There are two infinite classes left, then - the antiprisms and trapezohedra.
We expect that similar ``bootstrap'' lemmas can be proven for these classes, so all that is left is to prove that these classes are locally Rupert and locally reverse Rupert. With such a bootstrap in place, these hypothetical theorems would cover the case of the Dodecahedron and Icosihedron, a pleasing bit of symmetry with the Octahedron and Cube. It is unclear at present whether the technology developed in this paper, in particular that of the $J_v$ function, will be sufficient to do this, but nonetheless we conjecture the following.

\begin{conjecture}
Trapezohedra are locally Rupert and antiprisms are locally reverse Rupert. Furthermore, there are ``bootstrap'' lemmas for these infinite classes, showing that every polyhedron $Q$ with a trapezal polygonal section is locally Rupert and every polyhedron $Q$ with an antiprism section is locally reverse Rupert.
\end{conjecture}

The reader who is interested in pursuing this conjecture should begin with the Cube and Octahedron to build their intuition for how these classes behave, since the Cube is a trapezohedron (over a regular skew hexagon) and the Octahedron is an antiprism over the triangle. Both polyhedra are well-behaved and well-understood.

In addition, this basic procedure --- proving a local lemma for a class and then showing that the class is locally (reverse) Rupert --- is perhaps a profitable approach for broader classes than just the antiprisms and trapezohedra.

\subsection{Duality}

Corollary~\ref{cor:dual} gives a nice symmetry - in many of the cases we cover, if a polyhedron is locally Rupert, then its dual is reverse locally Rupert. This relationship between Rupertness and duality was also noted by the authors in \cite{Steininger2021}. Choosing the correct sense of ``dual'' is critical to a conjecture of this shape, but based on the evidence and some intuition about Rupert and reverse Rupert, we conjecture the following.

\begin{conjecture}
A polyhedron $Q$ is locally Rupert if and only if its dual $D$ is reverse locally Rupert.
\end{conjecture}

This conjecture is shy of a complete duality conjecture, which would drop the word ``local'' in the above, but we are more hesitant about such a theorem since there is less structure without locality.

\subsection{Details on $J_v$}

The present paper gives an explicit formula for the function $J_v(a) = \pi(\rho_a^\delta\cdot v)$ in terms of spherical trigonometry and a co-ordinate system for the sphere. To prove our theorems, we only required a very crude analysis of $J_v$, but with an explicit form in place it is reasonable to expect that more could be gleaned from a closer analysis. In particular, extending Theorem~\ref{thm:dap} to include the triangle (and other trivial double-arch polygons) seems to be possible with this closer analysis. All convex polygons are either trivial or nontrivial double-arch,\footnote{The proof follows from taking the longest line $r$ between two vertices of the polygon. By maximality, this line splits the polygon into two arches, one of which may be trivial.} so such an extension would greatly simplify the statement of Theorem~\ref{thm:dap}.

\bibliographystyle{unsrt}
\bibliography{Rupert Passage Bibliography}

\end{document}